\documentclass{elsart}
\usepackage{etex}
 \usepackage[usenames,dvipsnames]{pstricks}
 \usepackage{epsfig}
 \usepackage{pst-grad} % For gradients
 \usepackage{pst-plot} 
 \usepackage{pgf}
\usepackage{float}
\usepackage{tikz}
\usepackage{pgfbaseplot}
\usepackage{booktabs}
\usepackage{subcaption}

\usepackage{pstricks, pst-node, pst-plot, pst-circ}
\usepackage{moredefs}
\usepackage{array} 
\usepackage{extarrows}
\usepackage[fleqn,tbtags]{mathtools}
\usepackage{mhsetup}
\usepackage{arydshln}
\usepackage{cite}

\usepackage{amssymb,amsmath,amscd, mathrsfs, msc}
\usepackage{cases}
\usepackage{verbatim}
\usepackage[a]{esvect}
\usepackage{ntheorem}
\newcommand{\nocontentsline}[3]{}
\newcommand{\tocless}[2]{\bgroup\let\addcontentsline=\nocontentsline#1{#2}\egroup}

\theoremstyle{plain}

\newtheorem{theorem}{Theorem}[section]

\newtheorem{lemma}[theorem]{Lemma}
\theoremstyle{definition}

\theoremstyle{remark}

\numberwithin{equation}{section} \theoremstyle{corollary}

\newenvironment{proof}[1][Proof]{\begin{trivlist}
\item[\hskip \labelsep {\bfseries #1}]}{\end{trivlist}}

\allowdisplaybreaks
 
\begin{document}
\begin{frontmatter}

\title{A model of regulatory dynamics with  threshold-type state-dependent delay  \thanksref{GRN}}
\thanks[GRN]{}

 \author[UTD]{Qingwen Hu}\ead{qingwen@utdallas.edu} 
 \address[UTD]{Department of Mathematical Sciences,  The University of Texas at Dallas, Richardson, TX, 75080}

\date{}

\maketitle

\begin{abstract}
We  model  intracellular regulatory dynamics with threshold-type state-dependent delay and investigate the effect of the state-dependent  diffusion time. A general model which is  an extension of the classical differential equation models with constant  or zero time delays  is developed to study the stability of steady state, the occurrence and stability of periodic oscillations in regulatory dynamics. Using the method of multiple time scales, we compute the normal form of the general model and show that the state-dependent diffusion time may  lead to both supercritical and subcritical Hopf bifurcations. Numerical simulations of the prototype model of Hes1 regulatory dynamics are  given to illustrate the general results.\end{abstract}


%\tableofcontents
 
 \end{frontmatter}
\thispagestyle{plain}
 
\section{Introduction} \label{Sec1}

An important goal of system biology is to understand the mechanisms that govern  the protein regulatory dynamics. The so-called feedback loop  is believed to be widely present in various eukaryotic cellular processes  and regulates the gene expression  of a broad class of intracellular proteins. It is usually understood   that in the feedback loop  transcription factors are regulatory proteins which activate transcription in eukaryotic cells; they act by binding to a specific DNA sequences in the nucleus, either activating or inhibiting the binding of RNA polymerase to DNA; mRNA is transcribed in the nucleus and is in turn translated in the cytoplasm. Such a feedback loop was described by  systems of ordinary differential equations in \cite{Goodwin, Griffith1968a, Griffith1968b, Tyson-1, Smith1987b}.

Many scholars have noticed that the   genetic regulatory dynamics are dependent on the intracellular transport of mRNA and protein \cite{Mahaffy-Pao1984, Smolen1999a}. In particular, models of genetic regulatory dynamics are developed to explain a variety of sustained periodic biological phenomena \cite{Smolen2000}. Hence the possible roles played by the time delays in the signaling pathways become an interesting research topic and we are interested   how time delays are associated with periodic oscillations.  Regulatory models with constant delays are incorporated in the work of \cite{Smolen1999a, WAN2009464}, among many others, where time delay is  used as a parameter and is tuned to observe stable periodic oscillations. It was   pointed out in \cite{Smolen2000} that if the simplification of constant  time delay is too drastic, another approach is to assume a distributed time delay. In the work of Busenburg and Mahaffy \cite{Mahaffy-Busenberg}, both diffusive  macromolecular transport represented as a distributed delay and active  macromolecular transport modeled as time delay are considered. It was demonstrated  that stable periodic oscillations  can occur when transport was slowed by increasing the time delay, and that oscillations are dampened if transport   was slowed by reducing the diffusion coefficients. 

We follow the idea of the work of \cite{Mahaffy-Busenberg} to consider both diffusion transport  and active transport in genetic regulatory dynamics, while we start from the Goodwin's model \cite{Goodwin} in contrast to the  reaction-diffusion equations  in \cite{Mahaffy-Busenberg}. Namely, we will provide a new perspective to investigate genetic regulatory dynamics which generalizes the previous extension of Goodwin's model with constant delay and simplifies the approach adopted by  \cite{Mahaffy-Busenberg}. To be more specific, we  begin with the prototype model for Hes1 system (see Section~\ref{Sec2} for a brief introduction) and discuss how to  model the diffusion time if we consider the effect of the fluctuation of concentrations of the substances in the  Hes1 system. We show that modeling the diffusion time will lead to a model with threshold type distributed delay, which is also called threshold type state-dependent  delay  \cite{smith1993, BHK}.  With a general model with  threshold type state-dependent  delay,  we are interested how the model with state-dependent delay   differs from those with only constant delays, how the stabilities of the steady state and the periodic oscillations depend on the state-dependent delay.   

We organize  the remaining of the paper as follows. In Section~\ref{Sec2}, we model the diffusion time in regulatory dynamics and  set up a prototype system of differential equations with state-dependent delay for the Hes1 system. In Section~\ref{Sec3}, we consider a general model with state-dependent delay and conduct linear stability analysis of the equivalent model and conduct a  local Hopf bifurcation  analysis of the system using the multiple time scale method.  Numerical simulations on the model of Hes1 system with state-dependent delay will be illustrated in Section~\ref{Sec6}. We conclude the investigation at Section~\ref{Sect5}.

\section{State-dependent diffusion time in regulatory dynamics}\label{Sec2}
In this section, we motivate the modeling of diffusion time delay with the  Hes1 regulatory system. Hes1 is a  protein which belongs to the basic helix-loop-helix (bHLH,  a protein structural motif) family of transcription factors. It is a transcriptional repressor of genes that require a bHLH protein for their transcription. As a member of the bHLH family, it is a transcriptional repressor that influences cell proliferation and differentiation in embryogenesis.  Hes1 regulates its own expression via a negative feedback loop, and oscillates with approximately 2-hour periodicity \cite{Hirata}. However, the molecular mechanism of such oscillation remains to be determined \cite{Hirata}.  Mathematical modeling of Hes1 regulatory system has been an active research area in the last decade. See among many others, \cite{Sturrock,Jensen}.

% Generated with LaTeXDraw 2.0.8
% Sat Nov 10 11:29:01 CST 2012
% \usepackage[usenames,dvipsnames]{pstricks}
% \usepackage{epsfig}
% \usepackage{pst-grad} % For gradients
% \usepackage{pst-plot} % For axes
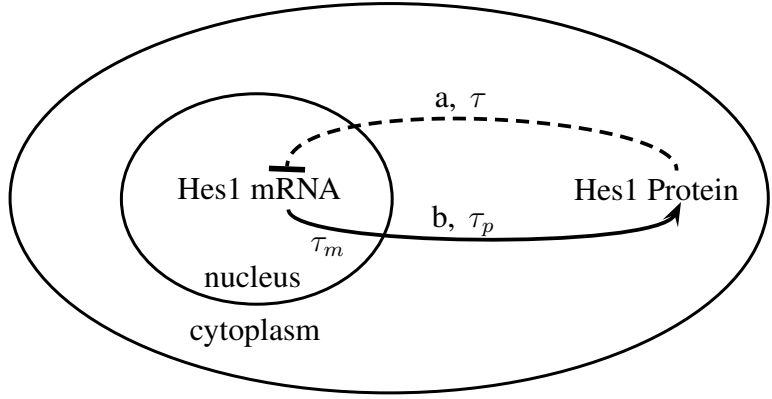
\begin{figure}\centering
% Generated with LaTeXDraw 2.0.8
% Sat Nov 10 13:26:35 CST 2012
% \usepackage[usenames,dvipsnames]{pstricks}
% \usepackage{epsfig}
% \usepackage{pst-grad} % For gradients
% \usepackage{pst-plot} % For axes
\scalebox{1} % Change this value to rescale the drawing.
{
\begin{pspicture}(0,-2.61)(10.12,2.61)
\psellipse[linewidth=0.04,dimen=outer](5.06,0.0)(5.06,2.61)
\psellipse[linewidth=0.04,dimen=outer](3.3,-0.01)(1.82,1.42)
\usefont{T1}{ptm}{m}{n}
\rput(8.616133,0.115){Hes1 Protein}
\usefont{T1}{ptm}{m}{n}
\rput(3.3348925,0.135){Hes1 mRNA}
\usefont{T1}{ptm}{m}{n}
\rput(3.2730665,-1.785){cytoplasm}
\usefont{T1}{ptm}{m}{n}
\rput(3.242666,-1.045){nucleus}
\rput(4.242666,-0.645){$\tau_m$}
\usefont{T1}{ptm}{m}{n}
\rput(6.013125,1.275){a,\, $\tau$}
\usefont{T1}{ptm}{m}{n}
\rput(6.044199,-0.305){b,\, $\tau_p$}
\psbezier[linewidth=0.05,linecolor=black,arrowsize=0.05291667cm 2.0,arrowlength=1.4,arrowinset=0.4]{<-}(8.94,-0.05)(8.62,-0.69)(3.7402325,-0.6558536)(3.72,-0.15)
\psbezier[linewidth=0.05,linestyle=dashed,linecolor=black,tbarsize=0.07055555cm 5.0]{-|}(8.88,0.3686207)(8.94,1.1134483)(3.74,1.43)(3.7,0.35)
\end{pspicture} 
}
\caption{Hes1 regulatory system in a cell: a. inhibition of mRNA transcription in nucleus from protein diffused from cytoplasm, b. translation of mRNA for protein synthesis in cytoplasm}\label{Figure1}
\end{figure}
The basic reaction kinetics for the regulatory system (Figure 1) can be described by the following ordinary differential equations which is in general called the Goodwin model (See \cite{Griffith1968a,Griffith1968b}). Denote by $x_m(t)$ and $y_p(t)$ the concentrations of the species of  Hes1 mRNA and Hes1 protein in a cell, respectively.  
\begin{align}\label{ODE-hes1}
\left\{
\begin{aligned}
\frac{\mathrm{d}x_m(t)}{\mathrm{d}t} & =-\mu_m x_m(t)+\frac{\alpha_m}{1+\left(\frac{y_p(t)}{\bar{y}}\right)^h},\\
  \frac{\mathrm{d}y_p(t)}{\mathrm{d}t} & =-\mu_p y_p(t)+\alpha_p x_m(t),
\end{aligned}
\right.
\end{align}
where $\mu_m$ and $\mu_p$ are the degradation rates of mRNA and protein, respectively, $\alpha_m$ the basal rate of initial transcription without feedback from protein, and $\alpha_p$ the production rate of protein, $\bar{y}$ the concentration of mRNA that reduces the rate of initial transcription to half of its basal value. 

 Since  mRNA is typically transcribed in nucleus and then translated to cytoplasm for protein synthesis, there exists transcriptional and translational delays in the real reaction kinetics. For a better description of the intracellular processes, (\ref{ODE-hes1}) was extended into  the following model of delay differential equation (see, e.\,g.,\cite{Jensen})
 \begin{align}\label{DDE-hes1}
\left\{
\begin{aligned}
\frac{\mathrm{d}x_m(t)}{\mathrm{d}t} & =-\mu_m x_m(t)+\frac{\alpha_m}{1+\left(\frac{y_p(t-\tau_m)}{\bar{y}}\right)^h},\\
  \frac{\mathrm{d}y_p(t)}{\mathrm{d}t} & =-\mu_p y_p(t)+\alpha_p x_m(t-\tau_p),
\end{aligned}
\right.
\end{align}
where $\tau_m$ and $\tau_p$ are transcriptional and translational delays, respectively. This model was able to simulate  oscillatory solutions with the period and mRNA expression lag in good agreement with some experimental data. The  transcriptional and translational  time delay are usually assumed to be constant  and it is then a natural  problem to model the time delays in order to investigate the effect of the varying time delays on genetic regulatory systems.

For simplicity, we combine the reaction or transcription time with the translation time between the nucleus and the cytoplasm  and call the combination is the diffusion time which will be denoted by $\tau$. 
  In such a way $\tau_m$ in (\ref{DDE-hes1}) is re-defined to be the mRNA transcription time and   $\tau_p$ the protein diffusion time. That is, both are assumed to be the part of the time for diffusion processes, and the termination of the diffusion process incurs that of the inhibition of mRNA. That is, we assume  that
\begin{enumerate}
\item[(A1)]  $\tau=\tau_m+\tau_p$.
\end{enumerate} 
Now we drop the subscripts for the variables $x_m,\,y_p,\,\tau_m$ and $\tau_p$ in the following presentation, and turn to the modeling of the  part of diffusion time due to spatial translation.  Let  $T$ to be the time to homogenize a newly produced solute (mRNA or protein) in the cell solvent, $\epsilon$ be the transcription time which will be assumed to be a constant.  From Fick's first law of diffusion, we know that the diffusion time for a newly produced solute to homogenize in a cell  is not necessarily a constant but is dependent on the concentration gradient.  To model the diffusion time delay, let us suppose that we put a drop (or subtract a drop) of solute in a cup of solvent with concentration $x(t)$ which has taken account of the new drop. The diffusion time $T$ to homogenize depends on the homogenized concentration $x(t-\tau)$ of the solute in the cup without taking account of the new drop. Therefore the diffusion time to homogenize the new drop depends on the concentration difference $x(t)-x(t-\tau)$.

The next problem is how $T$ depends on the concentration difference $x(t)-x(t-\tau)$. In order to derive a reasonable modeling of $\tau$, we virtually assume that the  distance of the central locations of concentrations $x(t)$ and $x(t-\tau(t))$ is $L$, where $L>0$ is small.  It is known that the diffusion time $\tau$ is inversely proportional to the diffusion coefficient ($D$) and approximately satisfies
$T=\frac{L^2}{2D}$. Then by Fick's first law we know that the diffusion flux (${J}$) is proportional to the existing concentration gradient under the assumption of steady state. Namely, $J$ can be approximated by $J=-D \frac{x(t)-x(t-\tau)}{L}$, where the positivity of $J$ is opposite to that of $x(t)-x(t-\tau)$.  To obtain an approximation of $\tau$  we reduce $D$   from the aforementioned expressions of $T$ and $J$. Namely,  $T $ can be approximated through
\begin{align}
T&=-\frac{L(x(t)-x(t-\tau))}{2J}.\label{Travel-time}
\end{align}Then we have
$\tau=T+\epsilon=-\frac{L(x(t)-x(t-\tau))}{2J}+\epsilon$.   In the following, we assume that
\begin{enumerate}
\item[(A2)] there exist  constants $\epsilon>0, \,c>0$ such that $\tau(t)=c (x(t)-x(t-\tau(t)))+\epsilon$. 
\end{enumerate}
Namely, we assume that the diffusion time is linearly dependent on the fluctuation  of the concentrations of mDNA. We remark that $x(t)-x(t-\tau)$ may be negative which means that current concentration is less than the historical one and we interpret that in this case the   transcription time is reduced by $c(x(t)-x(t-\tau))$ from $\epsilon$ with less concentration.
Moreover, if the concentrations are in their equilibrium states, the diffusion time becomes the transcriptional delay $\epsilon$.
  We also remark that such type of state-dependent delay has appeared in   position control \cite{MR2016913} modeling the state-dependent traveling time of the echo, and in regenerative turning processes \cite{Insperger2007} modeling the cutting time of one round of turning.
 
 Now  we obtain the following state-dependent delay differential equations,
\begin{align}\label{SDDE-hes1}
\left\{
\begin{aligned}
\frac{\mathrm{d}x(t)}{\mathrm{d}t} & =-\mu_m x(t)+\frac{\alpha_m}{1+\left(\frac{y(t-\tau)}{\bar{y}}\right)^h},\\
  \frac{\mathrm{d}y(t)}{\mathrm{d}t} & =-\mu_p y(t)+\alpha_p x(t-\tau),\\
 \tau(t)&=c(x(t)-x(t-\tau))+\epsilon.
\end{aligned}
\right.
\end{align}
which provides another view on the models (\ref{ODE-hes1}) and (\ref{DDE-hes1}): If we assume that the transcription and translation of the newly produced solute is instant, namely, $T=0$, then we have  model  (\ref{ODE-hes1}) which is a system of ordinary differential equation.  If we assume that the transcription and translation of the newly produced solute  needs a constant positive time, then we have model (\ref{DDE-hes1}) which is a system of delay differential equation with constant delay. 

It was shown in \cite{HKT} that system with the state-dependent delay in the form of $\tau(t)=r(x(t)-x(t-\tau(t)))$ with  $r$ being inhomogeneous linear can be transformed into models with both discrete and distributed delay which greatly facilitated the qualitative analysis of such type of models \cite{HKT,BHK}.   In the next section we  investigate systems~(\ref{SDDE-hes1})    based on the discussion of a more broad class of equations. 

\section{Regulatory dynamics  with state-dependent delay}\label{Sec3}
Regarding system~(\ref{SDDE-hes1})  as a prototype model of a range of genetic regulatory dynamics, we consider the following system, 
\begin{align}\label{SDDE-general-1}
\left\{
\begin{aligned}
\frac{\mathrm{d}x(t)}{\mathrm{d}t} & =-\mu_m x(t)+f(y(t-\tau)),\\
  \frac{\mathrm{d}y(t)}{\mathrm{d}t} & =-\mu_p y(t)+ g(x(t-\tau)),\\
 \tau(t)&=\epsilon+c(x(t)-x(t-\tau)),
\end{aligned}
\right.
\end{align}where $f,\,g: \mathbb{R}\rightarrow\mathbb{R}$ are three times continuously differentiable functions; $\mu_m$, $\mu_p$, $c$ and $\epsilon$ are positive constants. In the following we denote by $C([a,\,b];\mathbb{R}^N)$ with $a<b$ the space of continuously functions  $f: [a,\,b]\rightarrow\mathbb{R}^N$ equipped with the supremum norm  and by $C^1([a,\,b];\mathbb{R}^N)$ with $a<b$ the space of continuously differentiable functions.  We assume that 
\begin{enumerate}
\item[B1)]  System (\ref{SDDE-general-1}) has at least one equilibrium  $(r^*,\,\xi^*)$ for $(x,\,y)$ and there exists $\alpha_0>0$ and a  neighborhood $D$ of  $(r^*,\,\xi^*)$ in the space of continuous differentiable functions $C^1([-\alpha_0,\,0];\mathbb{R}^2)$  such that for every   initial data in $D$ for $(x,\,y)$ and the initial data $\tau_0\in(0,\,\alpha_0)$ satisfying the compatibility conditions
\begin{align}\label{compatibility}
\left\{
\begin{aligned}
x'(0) & =-\mu_m x(0)+f(y(-\tau_0)),\\
y'(0) & =-\mu_p y(0)+ g(x(-\tau_0)),\\
 \tau_0 & = \epsilon+c(x(0)-x(-\tau_0)),
\end{aligned}
\right.
\end{align}
  system (\ref{SDDE-general-1}) has a unique solution $(x,\,y,\,\tau)$ on $[0,\,+\infty)$. 
\item[B2)] Let $D$ be as in (B1).  For every solution  $(x,\,y,\,\tau)$ of system (\ref{SDDE-general-1})  on $[0,\,+\infty)$ with initial data in $D$ for $(x,\,y)$ and initial data $\tau_0\in(0,\,+\infty)$, we have $1/c >\dot{x}(t)$ for all  $t>0$. 
\end{enumerate}
We rewrite the   equation  for the delay in (\ref{SDDE-hes1}) as
\begin{align}
\int_{t-\tau(t)}^t\frac{1-c\dot{x}(s)}{\epsilon}\mathrm{d}s=1.\label{SDDE-eqn-5}
\end{align}
Let $u(t)=\dot{x}(t)$  and introduce the following  transformation inspired by \cite{smith1993}:
\begin{align}\label{smith-transform}
\eta  =\int_{0}^t\frac{1-cu(s)}{\epsilon}\mathrm{d}s, \quad
 r(\eta) =x(t),\quad
 \xi(\eta) =y(t),\quad
 k(\eta) = \tau(t).
\end{align} 
Then we have 
\begin{align} \label{const-delay} 
\left\{
\begin{aligned}
\eta-1 &= \int_{0}^{t-\tau(t)}\frac{1-cu(s)}{\epsilon}\mathrm{d}s,\\
 r(\eta-1) &=x(t-\tau(t)),\\
 \xi(\eta-1) &=y(t-\tau(t)),\\
 k(\eta) &=\epsilon+c(r(\eta)-r(\eta-1)).
 \end{aligned}
 \right.
\end{align}
Note that we have
\begin{align}\label{derive}
\left\{
\begin{aligned}
\dot{x}(t) & = \frac{\mathrm{d}r}{\mathrm{d} \eta} \cdot \frac{\mathrm{d}\eta}{\mathrm{d}t}=\dot{r}(\eta)\cdot\frac{1-c\dot{x}}{\epsilon},\\
\dot{y}(t) &= \frac{\mathrm{d}\xi}{\mathrm{d} \eta} \cdot \frac{\mathrm{d}\eta}{\mathrm{d}t}=\dot{\xi}(\eta)\cdot\frac{1-c\dot{x}}{\epsilon},
\end{aligned}
\right.
\end{align}
 Then by (\ref{const-delay}) and  (\ref{derive}), system  (\ref{SDDE-hes1}) is transformed into
 \begin{align}\label{tranformed-general}
\left\{
\begin{aligned}
 \frac{\mathrm{d}r}{\mathrm{d} \eta} & = \epsilon\frac{-\mu_m r(\eta)+f(\xi(\eta-1))}{1-c(-\mu_m r(\eta)+f(\xi(\eta-1)))},\\
  \frac{\mathrm{d}\xi}{\mathrm{d} \eta} & = \epsilon\frac{-\mu_p \xi(\eta)+g(r(\eta-1))}{1-c(-\mu_m r(\eta)+f(\xi(\eta-1)))},\\
 k(\eta) &=\epsilon+c(r(\eta)-r(\eta-1)),
\end{aligned}
\right.
\end{align}where the assumption that $1/c>\dot{x}(t)$ for all $t\in\mathbb{R}$ has been converted by the first equation in (\ref{derive}) into $\dot{r}(\eta)>-\epsilon/c$. We remark that if we set $c=0$ in system (\ref{tranformed-general}), we obtain
 \begin{align}\label{tranformed-general-c-0}
\left\{
\begin{aligned}
\frac{1}{ \epsilon} \frac{\mathrm{d}r}{\mathrm{d} \eta} & ={-\mu_m r(\eta)+f(\xi(\eta-1))}\\
 \frac{1}{ \epsilon}  \frac{\mathrm{d}\xi}{\mathrm{d} \eta} & = {-\mu_p \xi(\eta)+g(r(\eta-1))}\\
 k(\eta) &=\epsilon,
\end{aligned}
\right.
\end{align}
which is equivalent to (\ref{SDDE-general-1}) with $c=0$:
\begin{align}\label{SDDE-general-c-0}
\left\{
\begin{aligned}
\frac{\mathrm{d}x(t)}{\mathrm{d}t} & =-\mu_m x(t)+f(y(t-\tau)),\\
  \frac{\mathrm{d}y(t)}{\mathrm{d}t} & =-\mu_p y(t)+ g(x(t-\tau)),\\
 \tau(t)&=\epsilon,
\end{aligned}
\right.
\end{align} in the sense of the  time-domain transformations leading to system (\ref{tranformed-general}).

Comparing system (\ref{tranformed-general}) with  system (\ref{tranformed-general-c-0}), we see that both systems have the same set of equilibria. We also observe that,  if $1-c\dot{x}(t)>0$ is uniformly bounded above for $t>0$, then (\ref{SDDE-eqn-5}) describes the situation that the diffusion process is slow when $\epsilon$ is large and  is fast when $\epsilon$ is small. Moreover, from system (\ref{tranformed-general-c-0}) we know that the parameter  $\epsilon$  may serve as a time scale which is  highly desirable  to investigate the intercelluar regulatory dynamics. In the subsequent sections we study how the  dynamics of system  (\ref{tranformed-general}) varies with respect to the parameter $\epsilon\in (0,\,+\infty)$. We remark that limiting profiles of periodic solutions as $\epsilon\rightarrow 0$ was  investigated in \cite{Nussbaum1} for the equation
$
\epsilon\dot{x}(t)=-x(t) + f(x(t-1)).
$
 \subsection{Linear stability analysis}
Now we study the stability of the equilibrium of system (\ref{tranformed-general}). Note that the variable $k$ has been decoupled from the differential equations. Therefore, we assume that $(r^*,\,\xi^*)\in\mathbb{R}^2$ is the equilibrium of system (\ref{tranformed-general})   such that
\begin{align*}
\left\{
\begin{aligned}
-\mu_m r^*+f(\xi^*))&=0,\\
-\mu_p\xi^*+g(r^*)&=0.
\end{aligned}
\right.
\end{align*}
We translate the equilibrium of (\ref{tranformed-general}) to the origin by letting $(u,\ v)$ be such that  $r=u+r^*$, $\xi=v+\xi^*$ and let  $F,\,G: \mathbb{R}\rightarrow\mathbb{R}$ satisfy  that
\begin{align*}
F(v(\eta-1))=f(v(\eta-1)+\xi^*)-f(\xi^*),\quad
G(u(\eta-1))=g(u(\eta-1)+r^*)-g(r^*).
\end{align*}
Then we have $F'(0)=f'(\xi^*)$,  $G'(0)=g'(r^*)$, and
\begin{align}\label{eqn-5-2}
\left\{
\begin{aligned}
\frac{1}{\epsilon}\frac{\mathrm{d}u}{\mathrm{d}\eta}&=\frac{-\mu_m u+F(v(\eta-1))}{1-c[-\mu_m u+F(v(\eta-1))]},\\
\frac{1}{\epsilon}\frac{\mathrm{d}v}{\mathrm{d}\eta}&=\frac{-\mu_p v+G(u(\eta-1))}{1-c[-\mu_m u+F(v(\eta-1))]}.
\end{aligned}
\right.
\end{align}
Linearization of (\ref{eqn-5-2}) at $(0,\,0)$  with $x=(u,\,v)$ is
\begin{align}\label{linearized}
\dot{x}(\eta)=\epsilon
M
x(\eta)+\epsilon
Nx(\eta-1),
\end{align}
where $M=\begin{bmatrix}
-\mu_m & 0\\
0 & -\mu_p
\end{bmatrix}$ and $N=\begin{bmatrix}
 0 & f'(\xi^*)\\
g'(r^*) & 0 
\end{bmatrix}$.
The characteristic equation of (\ref{linearized}) is $\det(\lambda I+\epsilon M-\epsilon N e^{-\lambda})=0,$ which is equivalent to
\begin{align}\label{charac-1}
(\lambda+\epsilon\mu_m)(\lambda+\epsilon\mu_p)-\epsilon^2 f'(\xi^*)g'(r^*)e^{-2\lambda}=0.
\end{align}
Let $z=2\lambda$, $\tau=2\epsilon$ and $f'(\xi^*)g(r^*)=-h$, then the characteristic equation (\ref{charac-1}) becomes
\begin{align}\label{charac-2}
z^2+\tau(\mu_m+\mu_p)z+\tau^2\mu_m\mu_p+\tau^2h e^{-z}=0,
\end{align}
which is the same equation investigated in \cite{chenwu}.  Then by the results in \cite{chenwu} we have,
  \begin{lemma}\label{lemma-1}Consider  the characteristic equation (\ref{charac-1}), where $\mu_m$ and $\mu_p$ are positive constants. Then for every $\epsilon\in (0,\,+\infty)$,  the equation
  \begin{align*}
\beta^2-{\epsilon}^2 (\mu_m\mu_p)=\epsilon(\mu_m+\mu_p)\beta \cot2\beta,
\end{align*}has a unique solution for $\beta$ in $(0,\,\frac{\pi}{2})$, denoted by $\beta(\epsilon)$
and the following statements hold.
 \begin{enumerate}
\item[i)]If 
 $
 \mu_m\mu_p\geq -f'(\xi^*)g'(r^*)\geq 0,
 $  then the equilibrium of (\ref{tranformed-general}) is asymptotically stable.
\item[ii)] If  
 $
 \mu_m\mu_p< -f'(\xi^*)g'(r^*),
 $ then there exists a unique $\epsilon_0\in (0,\,+\infty)$ such that
 $
 (\mu_m+\mu_p) \beta(\epsilon_0)+\epsilon_0 f'(\xi^*)g'(r^*)\sin 2\beta(\epsilon_0)=0,
 $  and $\pm i\beta(\epsilon_0)$ is a pair of purely imaginary eigenvalues. Moreover,
  for every $\epsilon\in (0,\,\epsilon_0)$,  the equilibrium of (\ref{tranformed-general}) is asymptotically stable;   for every $\epsilon\in (\epsilon_0,\,+\infty)$,  the equilibrium of (\ref{tranformed-general}) is unstable, and there exists a sequence $\{\epsilon_k\}_{k=0}^\infty$ of critical values of $\epsilon$ for which  (\ref{charac-1}) has a corresponding sequence of purely imaginary eigenvalues  $\{i(\beta(\epsilon_0)+k\pi)\}_{k=1}^\infty$ where
 $
 \epsilon_k=\frac{\epsilon_0(\beta(\epsilon_0)+k\pi)}{\beta(\epsilon_0)}\in S_k,
 $ with $S_0 =\left\{\lambda\in\mathbb{C}: \mathrm{Im}{\lambda}\in (0,\,\frac{\pi}{2})\right\},$ $ S_k =\left\{\lambda\in\mathbb{C}: \mathrm{Im}{\lambda}\in ((k-\frac{1}{2})\pi,\,(k+\frac{1}{2})\pi)\right\}, \,k=1,\,2\cdots.$
\end{enumerate}
\end{lemma}
 \begin{lemma} Consider  the characteristic equation (\ref{charac-1}), where $\mu_m$ and $\mu_p$ are positive constants.  If  
 \[
 \mu_m\mu_p< -f'(\xi^*)g'(r^*),
 \] then there exists a unique $\epsilon_0\in (0,\,+\infty)$ such that
 \[
 (\mu_m+\mu_p) \beta(\epsilon_0)+\epsilon_0 f'(\xi^*)g'(r^*)\sin 2\beta(\epsilon_0)=0,
 \]  and $\pm i\beta(\epsilon_0)$ is a pair of purely imaginary eigenvalues.  Moreover, $\beta(\epsilon_0)$ satisfies
 \begin{align*}
 \tan 2\beta(\epsilon_0)=\frac{\sqrt{l}(\mu_m+\mu_p)}{1-l\mu_m\mu_p},
 \end{align*}
 where
  \begin{align*}
 l=\frac{\mu_m^2+\mu_p^2+\sqrt{(\mu_m^2-\mu_p^2)^2+4f'^2(\xi^*)g'^2(r^*)} }{2(f'^2(\xi^*)g'^2(r^*)-\mu_m^2\mu_p^2)},\quad
 \epsilon_0=\sqrt{l}\beta(\epsilon_0).
 \end{align*}
 \end{lemma}
 \begin{proof}By Lemma~\ref{lemma-1} $\beta(\epsilon_0)$ satisfies
 \begin{align}\label{eqn-tan-epsilon}
 \left\{
 \begin{aligned} \beta^2(\epsilon_0)-{\epsilon_0}^2 (\mu_m\mu_p)-\epsilon_0(\mu_m+\mu_p)\beta(\epsilon_0) \cot2\beta(\epsilon_0)&=0,\\
 (\mu_m+\mu_p) \beta(\epsilon_0)+\epsilon_0 f'(\xi^*)g'(r^*)\sin 2\beta(\epsilon_0)&=0.
 \end{aligned}
  \right.
 \end{align}
  Eliminating the trigonometric terms in (\ref{eqn-tan-epsilon}), we obtain that
  \[
  \epsilon_0^4(\mu_m^2\mu_p^2-f'^2(\xi^*)g'^2(r^*)))+\epsilon_0^2(\mu_m^2+\mu_p^2)\beta^2(\epsilon_0)+\beta^4(\epsilon_0)=0,
  \]
which can be regarded as a quadratic equation of $\epsilon_0^2$ and has a unique nonnegative root since $ \mu_m\mu_p\leq -f'(\xi^*)g'(r^*).$ We have
 \begin{align*}
 \epsilon_0^2=l\beta^2(\epsilon_0),\quad
  l=\frac{\mu_m^2+\mu_p^2+\sqrt{(\mu_m^2-\mu_p^2)^2+4f'^2(\xi^*)g'^2(r^*)} }{2(f'^2(\xi^*)g'^2(r^*)-\mu_m^2\mu_p^2)}.
 \end{align*}
Bringing $\epsilon_0=\sqrt{l}\beta(\epsilon_0)$ into the first equation of (\ref{eqn-tan-epsilon}), we have
\begin{align*}
 \tan 2\beta(\epsilon_0)=\frac{\sqrt{l}(\mu_m+\mu_p)}{1-l\mu_m\mu_p}. \qed
 \end{align*}
 \end{proof}

\subsection{Local Hopf bifurcation}\label{Sec4}
In this section we use multiple time scale method to find the normal form of system (\ref{tranformed-general}) at the first critical value $\epsilon^*$  of $\epsilon$.  Namely, we let $\epsilon^*=\epsilon_0$ where $\epsilon_0$ is defined at Lemma~4.1. Since we are interested in the dynamics when $\epsilon$ varies in $(0,\,\infty)$, we study the directions of 
the Hopf bifurcation and the stability of the bifurcating periodic solutions near the equilibrium. We first check the transversality condition.  Put $\lambda=\alpha+i\beta$ with $\alpha,\,\beta\in\mathbb{R}$ in   the characteristic equation (\ref{charac-1}) and we obtain
\begin{align}\label{re-um-3}
\begin{cases}
\begin{aligned}
\alpha^2-\beta^2+\epsilon(\mu_m+\mu_p)\alpha+\epsilon^2\mu_m\mu_p-\epsilon^2f'(\xi^*)g'(r^*)e^{-2\alpha}\cos 2\beta&=0,\\
2\alpha\beta+ \epsilon(\mu_m+\mu_p)\beta+\epsilon^2f'(\xi^*)g'(r^*)e^{-2\alpha}\sin2\beta&=0.
\end{aligned}
\end{cases}
\end{align}
By Implicit function theorem, we can show that $\alpha$ and $\beta$ are continuously differentiable with respect to $\epsilon$ near $(\alpha,\,\beta,\,\epsilon)=(0,\,\beta(\epsilon_0),\,\epsilon_0)$. We take derivatives respect to $\epsilon$ on both sides of each of the equations at (\ref{re-um-3}) and then let $\alpha=0$. We obtain,
\begin{align}\label{re-um-2}
\left\{
\begin{aligned}
 \frac{\mathrm{d}\alpha}{\mathrm{d} \epsilon} & \left(\epsilon^*(\mu_m+\mu_p)+2{\epsilon^*}^2 f'(\xi^*)g'(r^*)\cos 2\beta\right)+ \frac{\mathrm{d}\beta}{\mathrm{d} \epsilon}\left(-2\beta+2{\epsilon^*}^2 f'(\xi^*)g'(r^*)\sin 2\beta\right)\\ &=-2\epsilon^*\mu_m\mu_p+2f'(\xi^*)g'(r^*) \cos2\beta,\\
  \frac{\mathrm{d}\alpha}{\mathrm{d} \epsilon} & \left(2\beta-2{\epsilon^*}^2 f'(\xi^*)g'(r^*)\sin 2\beta\right)+ \frac{\mathrm{d}\beta}{\mathrm{d} \epsilon} \left(\epsilon^*(\mu_m+\mu_p)+2{\epsilon^*}^2 f'(\xi^*)g'(r^*)\cos 2\beta\right)\\ &=-(\mu_m+\mu_p)\beta-2\epsilon^*f'(\xi^*)g'(r^*) \sin2\beta.
 \end{aligned}
\right.
\end{align}
Noticing that if $\alpha=0$, $\epsilon=\epsilon^*$,  (\ref{re-um-3}) implies that
\begin{align}\label{re-um-4}
\begin{cases}
{\epsilon^*}^2f'(\xi^*)g'(r^*) \cos2\beta & ={\epsilon^*}^2\mu_m\mu_p-\beta^2,\\
{\epsilon^*}^2f'(\xi^*)g'(r^*) \sin2\beta & =-{\epsilon^*}(\mu_m+\mu_p)\beta.
\end{cases}
\end{align}
Combining (\ref{re-um-2}) and (\ref{re-um-4}), we obtain that  
\begin{align}\label{phi-direction}
 \frac{\mathrm{d}\alpha}{\mathrm{d} \epsilon}\,\vline_{\epsilon=\epsilon^*,\,\lambda=\pm i\beta}=\frac{\frac{2\beta^2}{\epsilon^*}({\epsilon^*}^2(\mu_m^2+\mu_p^2)+2\beta^2)}{(\epsilon^*(\mu_m+\mu_p)+2{\epsilon^*}^2\mu_m\mu_p-2\beta^2)^2+(2\beta+2\beta\epsilon^*(\mu_m+\mu_p))^2}>0.
\end{align}
 To illustrate  the stability of system~(\ref{SDDE-general-1}), we have,
\begin{theorem}\label{thm3-3} Consider system~(\ref{SDDE-general-1}) with an equilibrium $(r^*,\,\xi^*)$, where $f,\,g: U\subset C^1([-\alpha_0,\,0];\mathbb{R}^2)\rightarrow\mathbb{R}$ are three times continuously differentiable and $U$ an open neighborhood of $(r^*,\,\xi^*)$; $\mu_m,\,\mu_p$ and $c$ are positive constants.   Assume (B1) and (B2) hold.   Then the following statements hold.
\begin{enumerate}
\item[i)]If $
 \mu_m\mu_p\geq -f'(\xi^*)g'(r^*)\geq 0,$  then the equilibrium of system~(\ref{SDDE-general-1}) is asymptotically stable.
\item[ii)] If  
$
 \mu_m\mu_p< -f'(\xi^*)g'(r^*),
 $ then there exists a unique $\epsilon_0\in (0,\,+\infty)$ such that 
  for every $\epsilon\in (0,\,\epsilon_0)$,  the equilibrium $(r^*,\,\xi^*)$  is asymptotically stable;  for every $\epsilon\in (\epsilon_0,\,+\infty)$,  the equilibrium $(r^*,\,\xi^*)$  is unstable, and there exists a sequence $\{\epsilon_k\}_{k=0}^\infty$ of critical values of $\epsilon$ for which  system~(\ref{SDDE-general-1}) undergoes Hopf bifurcation.
\end{enumerate}
\end{theorem}
\begin{proof} By the technique of formal linearization \cite{Cooke1996,Walther2004}, we know that the equilibrium stability of system~(\ref{SDDE-general-1}) is equivalent to that of system~(\ref{SDDE-general-c-0}) which has the same equilibrium and characteristic equation for system~(\ref{tranformed-general}).
%
%
%Let $(x,\,y)$ be a solution of system~(\ref{SDDE-general-1}) and $(r,\,\xi)$ the corresponding solution of system~(\ref{tranformed-general}). By the transformations at (\ref{smith-transform}) we know that
%\[
%\sup_{t\geq 0}\|(x(t),\,y(t))-(r^*,\,\xi^*)\|=\sup_{\eta \geq 0}\|(r(\eta),\,\xi(\eta))-(r^*,\,\xi^*)\|.
%\]
%We  note that the transformation at (\ref{const-delay}) give a correspondence between the initial data in $C([-\alpha_0,\,0];\mathbb{R}^2)$ for system~(\ref{tranformed-general}) and that in $C([-1,\,0];\mathbb{R}^2)$ for system~(\ref{SDDE-general-1}) in a small neighborhood of $(r^*,\,\xi^*)$, respectively. 
Then by Lemma~\ref{lemma-1},   the stability of   $(r^*,\,\xi^*)$ of system~(\ref{tranformed-general})  implies that of system~(\ref{SDDE-general-1}).  Existence of Hopf bifurcation follows from the Hopf bifurcation theorem established in \cite{BHK}.
\qed
\end{proof}

  Let $x(\eta)=(u(\eta),\,v(\eta))$ and system (\ref{tranformed-general})  is written
\[
\frac{1}{\epsilon}\dot{x}(\eta)=\begin{bmatrix}
S((x(\eta),x(\eta-1))\\
H((x(\eta),x(\eta-1))
\end{bmatrix},
\]
where $(S,\,H)$ denotes the right hand side of (\ref{eqn-5-2}).   
%To compute the Taylor expansion of $(S,\,\,H)$, we denote by $(x_1,\,x_2,\,x_3,\,x_4)=(u(\eta),v(\eta),\,u(\eta-1),\,v(\eta-1))$. 
% The process of computing Taylor expansion is removed but are available in the drafts after Feb 2016. It can be copied and pasted to recover.
We have the Taylor expansion of (\ref{eqn-5-2}) up to the cubic terms,
\begin{align}\label{Taylor}
\left\{
\begin{aligned}
\frac{1}{\epsilon}\frac{\mathrm{d}u}{\mathrm{d}\eta}=&-\mu_mu(\eta)+f'(\xi^*)v(\eta-1)+ c\mu_{m}^2 u^2(\eta)-2c\mu_mf'(\xi^*)u(\eta)v(\eta-1)\\
     &+\left(\frac{1}{2}f''(\xi^*)+cf'^2(\xi^*)\right)v^2(\eta-1)-c^2\mu_m^3  u^3(\eta)+3c^2\mu_m^2f'(\xi^*) u^2(\eta)v(\eta-1) \\
       & -(c\mu_m f''(\xi^*)+3c^2\mu_m f'^2(\xi^*))  u(\eta) v^2(\eta-1) \\ & + \left(\frac{1}{6}f'''(\xi^*)+cf''(\xi^*) f'(\xi^*)+c^2f'^3(\xi^*)\right)v^3(\eta-1),\\
\frac{1}{\epsilon}\frac{\mathrm{d}v}{\mathrm{d}\eta}=&-\mu_p v(\eta)+g'(r^*)u(\eta-1)+\frac{1}{2}\left[-2c\mu_p f'(\xi^*)v(\eta) v(\eta-1)+2c\mu_m\mu_p u(\eta)  v(\eta)\right.\\
         & \left.  +g''(r^*)u^2(\eta-1)+2cf'(\xi^*)g'(r^*)u(\eta-1)v(\eta-1)-2c\mu_m g'(r^*)u u(\eta-1) \right]\\
         &+\frac{1}{6}\left[g'''(r^*)  u^3(\eta-1)-6c^2\mu_m^2\mu_p  u^2(\eta) v(\eta)+6c^2\mu_m^2g'(r^*) u^2(\eta) u(\eta-1)\right.\\
         & \left.-3c\mu_mg''(r^*)u(\eta) u^2(\eta-1)+3cg''(r^*)f'(\xi^*)u^2(\eta-1)v(\eta-1)\right.\\
         & \left.
       -3(c\mu_p  f''(\xi^*)+2c^2\mu_pf'^2(\xi^*))  v(\eta) v^2(\eta-1)\right.\\
         & \left.
        +3(cf''(\xi^*)g'(r^*)+2c^2f'^2(\xi^*)g'(r^*))u(\eta-1)v^2(\eta-1)\right.\\
        &\left.+12c^2\mu_m\mu_p f'(\xi^*)u(\eta) v(\eta) v(\eta-1)) -12c^2\mu_mg'(r^*)f'(\xi^*)  u(\eta) u(\eta-1)v(\eta-1)  \right]. 
        \end{aligned}
\right.
\end{align}

\subsection{Normal form computation via multiple time scale}
In this section we compute the normal form of system~(\ref{eqn-5-2}) using its Taylor expansion at (\ref{Taylor}). We seek a uniform second-order approximate solution in the form
\begin{align}\label{mts-1}
x(\eta;\, s)=sx_1(T_0,\,T_2)+s^2x_2(T_0,\,T_2)+s^3x_3(T_0,\,T_2)+\cdots,
\end{align}
where $T_0=\eta$, $T_2=s^2\eta$, and $s$ is a nondimensional book keeping parameter. In the following, we write $x_i=(u_i,\,v_i),\,i=1,\,2,\,\cdots$. It is not necessary to relate the approximate solution with the time scale $T_1=s\eta$ because a second order approximation will require independence of $T_1$ in the solution (See, e.g., \cite{Nayfeh}, page 166).
\begin{align}\label{mts-2}
\frac{\mathrm{d}}{\mathrm{d}\eta}=\frac{\partial}{\partial T_0}+s^2\frac{\partial }{\partial T_2}+\cdots=D_0+s^2 D_2+\cdots.
\end{align}
By (\ref{mts-1}) and (\ref{mts-2}) we have  
\begin{align}\label{mts-3}
x(\eta-1)&=x(T_0-1,\,T_2-s^2) =\sum_{n=1}^\infty s^n x_n(T_0-1,\,T_2)-s^3 D_2 x_1(T_0-1,\,T_2)\cdots.
\end{align}
Let $\epsilon=\epsilon^*+s^2\delta$ where $\epsilon^*$ is the critical value for Hopf bifurcation and $\delta$ a small detuning parameter.  Bringing  (\ref{mts-1}) and (\ref{mts-3}) into  (\ref{Taylor}), we have the left hand side been transformed into
\begin{align*}
 &\frac{1}{(\epsilon^*+s^2\delta)}(D_0+s^2 D_2)\sum^{\infty}_{n=1}s^n x_n(T_0,\,T_2)
\\ = & \frac{1}{(\epsilon^*+s^2\delta)}\left[ s D_0 x_1(T_0,\,T_2)+s^2D_0x_2(T_0,\,T_2)\right.\\
& \left.+ s^3D_0 x_3(T_0,\, T_2)+ s^3 D_2 x_1(T_0,\,T_2)+s^4 D_2 x_2(T_0,\,T_2)+\cdots\right],
\end{align*}
and the right hand side becomes
\begin{align*}
M\sum^{\infty}_{n=1}s^n x_n (T_0,\,T_2) +N\left(\sum_{n=1}^\infty s^n x_n(T_0-1,\,T_2)-s^3D_2 x_1(T_0-1,\,T_2)+\cdots\right)+\cdots.
\end{align*} Comparing the like powers of $s$ up to the cubic in the transformed equation of  (\ref{Taylor})  yield the following equations, where for notational simplicity we omitted the time scales of the variables except for the delayed ones and use dashed lines to separate entries of the column matrices.
\begin{align}
  & D_0 x_1-\epsilon^*Mx_1-\epsilon^*Nx_1(T_0-1,\,T_2) =0,\label{s-eqn-1}  \\
  & D_0 x_2-\epsilon^*Mx_2-\epsilon^*Nx_2(T_0-1,\,T_2) =\notag \\
  & \quad\quad \frac{\epsilon^* }{2}
\begin{bmatrix*}[l]
& 2c\mu_m^2 u_1^2-4c\mu_m f'(\xi^*)u_1v_1(T_0-1, T_2)+[f''(\xi^*)+2cf'^2(\xi^*)]v_1^2(T_0-1, T_2)\\
      \hdashline
& 2cf'(\xi^*)g'(r^*)u_1(T_0-1)v_1(T_0-1)-2c\mu_mg'(r^*)u_1u_1(T_0-1)\\
 &\quad\quad-2c\mu_p f'(\xi^*)  v_1 v_1(T_0-1)+2c\mu_m\mu_p u_1 v_1 +g''(r^*)u_1^2(T_0-1)\\
\end{bmatrix*},\label{s-eqn-2}\\
 & D_0 x_3-\epsilon^{*}Mx_3-\epsilon^*Nx_3(T_0-1,\,T_2) = -D_2 x_1-\epsilon^*N D_2x_1(T_0-1,\,T_2)\notag\\
&\quad\quad  +\delta (Mx_1+Nx_1(T_0-1,\,T_2))\notag\\
        & \quad\quad+\frac{\epsilon^*}{6}
        \begin{bmatrix*}[l]
        -6c^2\mu_m^3u_1^3+18c^2\mu_m^2f'(\xi^*)u_1^2v_1(T_0-1,\,T_2)\\
        \quad\quad-(6c\mu_m f''(\xi^*)+18c^2\mu_m f'^2(\xi^*))u_1 v_1^2(T_0-1,T_2)\\
        \quad\quad+(f'''(\xi^*)+6cf''(\xi^*)f'(\xi^*)+6c^2f'^3(\xi^*))v_1^3(T_0-1,\,T_2)\\
        \hdashline
        g'''(r^*)  u_1^3(T_0-1, T_2)-6c^2\mu_m^2\mu_p  u_1^2  v_1 +6c^2\mu_m^2g'(r^*) u_1^2  u_1(T_0-1, T_2) \\
          \quad\quad  -3c\mu_mg''(r^*)u_1 u_1^2(T_0-1, T_2)\\
          \quad\quad +3cg''(r^*)f'(\xi^*)u_1^2(T_0-1, T_2)v_1(T_0-1, T_2) \\
       \quad\quad  
       -3(c\mu_p  f''(\xi^*)+2c^2\mu_pf'^2(\xi^*))  v_1 v_1^2(T_0-1, T_2) \\
        \quad\quad 
        +3(cf''(\xi^*)g'(r^*)+2c^2f'^2(\xi^*)g'(r^*))u_1(T_0-1, T_2)v_1^2(T_0-1, T_2) \\
         \quad\quad +12c^2\mu_m\mu_p f'(\xi^*)u_1 v_1 v_1(T_0-1, T_2) \\  \quad\quad-12c^2\mu_mg'(r^*)f'(\xi^*)  u_1 u_1(T_0-1, T_2)v_1(T_0-1, T_2)\\
        \end{bmatrix*}\notag\\
        &\quad\quad +\frac{\epsilon^*}{2} 
        \begin{bmatrix*}[l]
        2c\mu_m^2(u_1u_2+u_2u_1)-4c\mu_m f'(\xi^*)(u_2v_1(T_0-1, T_2)+u_1v_2(T_0-1, T_2))\\
      \quad\quad  +2(f''(\xi^*)+2cf'^2(\xi^*)))v_1(T_0-1, T_2)v_2(T_0-1)\\
            \hdashline
       -2c\mu_p f'(\xi^*)(v_1v_2(T_0-1, T_2)+v_2v_1(T_0-1, T_2))+2c\mu_m\mu_p(u_2v_1+u_1v_2)\\
      \quad\quad  +2g''(r^*)u_1(T_0-1,T_2)u_2(T_0-1, T_2)\\
      \quad\quad +2c f'(\xi^*)g'(r^*)(u_1(T_0-1, T_2)v_2(T_0-1, T_2)\\
      \quad\quad+u_2(T_0-1, T_2)v_1(T_0-1, T_2)))\\
       \quad\quad -2c\mu_m g'(r^*)(u_2u_1(T_0-1, T_2)+u_1u_2(T_0-1, T_2))
        \end{bmatrix*}.\label{s-eqn-1}
\end{align}

We know that when $\epsilon=\epsilon^*$ system~(\ref{eqn-5-2}) undergoes Hopf bifurcation and   (\ref{s-eqn-1}) has periodic solution of the form 
\[
x_1=A(T_2)\theta e^{iw^*T_0}+\bar{A}(T_2)\bar{\theta}e^{-iw^*T_0},
\] where $iw^*$ is a purely imaginary eigenvalue with eigenvector $\theta=[\theta_1,\,\theta_2]^T\neq 0$ which satisfies
\begin{align*}
\left(iw^*I+\epsilon^*\begin{bmatrix}
\mu_m & 0\\
0 & \mu_p \end{bmatrix}
-\epsilon^*\begin{bmatrix}
0 & f'(\xi^*)e^{-iw^*}\\
g'(r^*)e^{-iw^*} & 0
\end{bmatrix}\right)
\left(\begin{matrix}\theta_1\\ \theta_2 \end{matrix}\right)&=0.%\\
%\begin{bmatrix}
%iw^*+\epsilon^*\mu_m & -\epsilon^* f'(\xi^*)e^{-iw^*}\\
%-\epsilon^* g'(r^*)e^{-iw^*} & iw^*+\epsilon^*\mu_p
%\end{bmatrix}
%\left(\begin{matrix}\theta_1 \\ \theta_2 \end{matrix}\right)&=0
\end{align*}
Leting $\theta_1=1$ in the above equation we have
%\begin{align*}
%\left\{ 
%\begin{matrix}
% (i w^*+\epsilon^* \mu_m)-\epsilon^*f'(0)e^{-i w^*}\theta_2=0\\
%(i w^*+\epsilon^*\mu_p)\theta_2-\epsilon^* g'(0)e^{-i w^*}=0
%\end{matrix}
%\right.
%\end{align*} and
\begin{align*}
 \theta=\left[
                       \begin{matrix}
                       1\\
                       \frac{e^{i w^*}(i w^*+\epsilon^*\mu_m)}{\epsilon^* f'(\xi^*)}
                       \end{matrix}
                       \right].                   
\end{align*}
Substitute $x_1$ into (\ref{s-eqn-2}), we have 
\begin{align} \label{eqn-5-10}
& D_0 x_2-\epsilon^*Mx_2-\epsilon^*Nx_2(T_0-1,\,T_2)\notag\\
=&\frac{\epsilon^*}{2}\begin{bmatrix*}[l]
&2c\mu_m^2(A e^{i w^*T_0}+\bar{A}e^{-i w^*T_0})^2\\
&-4c\mu_m f'(\xi^*)(Ae^{iw^*T_0}+\bar{A}e^{-iw^* T_0}) (A\theta_2 e^{iw^*(T_0-1)}+\bar{A}\bar{\theta_2}e^{-iw^*(T_0-1)})\\
& +[f''(\xi^*)+2cf'^2(\xi^*)]\cdot (A\theta_2e^{iw^*(T_0-1)}+\bar{A}\bar{\theta_2}e^{-iw^*(T_0-1)})^2\\
    \hdashline
    &-2c\mu_pf'(\xi^*)(A\theta_2e^{iw^*T_0}+\bar{A}\bar{\theta_2}e^{-iw^*T_0})(A\theta_2e^{iw^*(T_0-1)}+\bar{A}\bar{\theta_2}e^{-iw^*(T_0-1)})\\
    &+2c\mu_m\mu_p(Ae^{iw^*T_0}+\bar{A}e^{-iw^*T_0})(A\theta_2e^{iw^*T_0}+\bar{A}\bar{\theta_2}e^{-iw^*T_0})\\
    &+g''(r^*)(Ae^{iw^*(T_0-1)}+\bar{A}e^{-iw^*(T_0-1)})^2\\
    &+2cf'(\xi^*)g'(r^*)(Ae^{iw^*(T_0-1)}+\bar{A}e^{-iw^*(T_0-1)})(A\theta_2e^{iw^*(T_0-1)}+\bar{A}\bar{\theta_2}e^{-iw^*(T_0-1)})\\
    &-2c\mu_mg'(r^*)(Ae^{iw^*T_0}+\bar{A}e^{-iw^*T_0})(Ae^{iw^*(T_0-1)}+\bar{A}e^{-iw^*(T_0-1)})
     \end{bmatrix*}.
\end{align}
Equation  (\ref{s-eqn-2}) has a particular solution of the form
\[x_2=aA^2e^{i2w^*T_0}+bA\bar{A}+\bar{a}\bar{A}^2e^{-i2 w^*T_0},\]
where $a=(a_1,\,a_2)\in\mathbb{C}^2,\,b=(b_1,\,b_2)\in\mathbb{C}^2.$
Bringing the undetermined form of $x_2$ into (\ref{eqn-5-10}),  we obtain
\begin{align}\label{eqn-5-12}
&(aA^2e^{i2w^*T_0}\cdot i2w^*+\bar{a}\bar{A}^2e^{-i2w^*T_0}(-i2w^*))\\
& +\epsilon^*\left[\begin{matrix} \mu_m & 0\notag\\
                        0 & \mu_p \end{matrix}\right]
     (aA^2e^{i2w^*T_0}+bA\bar{A}+\bar{a}\bar{A}^2e^{-i2w^*T_0})\notag\\
&-\epsilon^*\left[\begin{matrix} 0 & f'(\xi^*)\\
                           g'(r^*)  & 0 \end{matrix}\right]
     (aA^2e^{i2w^*(T_0-1)}+bA \bar{A}+\bar{a} \bar{A}^2 e^{-i2w^*(T_0-1)})\notag\\
     =& \frac{\epsilon^*}{2}\begin{bmatrix*}[l]
&2c\mu_m^2(A e^{i w^*T_0}+\bar{A}e^{-i w^*T_0})^2\\
&-4c\mu_m f'(\xi^*)(Ae^{iw^*T_0}+\bar{A}e^{-iw^* T_0}) (A\theta_2 e^{iw^*(T_0-1)}+\bar{A}\bar{\theta_2}e^{-iw^*(T_0-1)}) \\
&+[f''(\xi^*)+2cf'^2(\xi^*)]\cdot (A\theta_2e^{iw^*(T_0-1)}+\bar{A}\bar{\theta_2}e^{-iw^*(T_0-1)})^2\\
    \hdashline
    &-2c\mu_pf'(\xi^*)(A\theta_2e^{iw^*T_0}+\bar{A}\bar{\theta_2}e^{-iw^*T_0})(A\theta_2e^{iw^*(T_0-1)}+\bar{A}\bar{\theta_2}e^{-iw^*(T_0-1)})\\
    &+2c\mu_m\mu_p(Ae^{iw^*T_0}+\bar{A}e^{-iw^*T_0})(A\theta_2e^{iw^*T_0}+\bar{A}\bar{\theta_2}e^{-iw^*T_0})\\
    &+g''(r^*)(Ae^{iw^*(T_0-1)}+\bar{A}e^{-iw^*(T_0-1)})^2\\
    &+2cf'(\xi^*)g'(r^*)(Ae^{iw^*(T_0-1)}+\bar{A}e^{-iw^*(T_0-1)})(A\theta_2e^{iw^*(T_0-1)}+\bar{A}\bar{\theta_2}e^{-iw^*(T_0-1)})\\
    &-2c\mu_mg'(r^*)(Ae^{iw^*T_0}+\bar{A}e^{-iw^*T_0})(Ae^{iw^*(T_0-1)}+\bar{A}e^{-iw^*(T_0-1)})
     \end{bmatrix*}.
\end{align}
Equating the coefficients of the terms with $e^{i2w^*T_0}$  from both sides of (\ref{eqn-5-12}), we obtain
\begin{align}\label{eqn-5-13}
\left\{
\begin{aligned}
   & (i2w^*)A^2a_1+\epsilon^*\mu_mA^2a_1-\epsilon^*f'(\xi^*)A^2e^{-i2w^*}a_2 \\
   =  &  \epsilon^*c\mu_m^2A^2-2\epsilon^*c\mu_mf'(\xi^*)A^2\theta_2e^{-iw^*}   
+\frac{\epsilon^*}{2}(f''(\xi^*)+2cf'^2(\xi^*))A^2\theta_2^2e^{-i2w^*},\\
& (i2w^*)A^2a_2+\epsilon^*\mu_pA^2a_2-\epsilon^*g'(r^*)A^2e^{-i2w^*}a_1 \\
= &  -\epsilon^*c\mu_pf'(\xi^*)\theta^2_2A^2e^{-iw^*}+\epsilon^*c\mu_m\mu_pA^2\theta_2  +\frac{\epsilon^*}{2}g''(r^*)A^2e^{-2iw^*}\\
  & +\epsilon^*cf'(\xi^*)g'(r^*)A^2\theta_2e^{-2iw^*}  -\epsilon^*c\mu_mg'(r^*)A^2e^{-iw^*}.
    \end{aligned}
        \right.
\end{align}
Equation (\ref{eqn-5-13}) has a solution  for $a$ if $i2w^*$ is not an eigenvalue of (\ref{linearized}). 
\begin{enumerate}
\item[(B3)] $i2w^*$ is not an eigenvalue of (\ref{linearized}). 
\end{enumerate}
We have
\allowbreak
\begin{align}\label{eqn-b-1-2}
\left\{
\begin{aligned}
a_1= &\frac{\epsilon^*}{ (i2w^*+\epsilon^*\mu_m)(i2w^*+\epsilon^*\mu_p) -{\epsilon^*}^2f'(\xi^*)g'(r^*)e^{-i4w^*} }\\
&\times\bigg[  \left(  (c\mu_m^2-2c\mu_mf'(\xi^*)\theta_2e^{-iw^*} +\frac{1}{2}(f''(\xi^*)+2cf'^2(\xi^*))\theta_2^2e^{-i2w^*} \right)(i2w^*+\epsilon^*\mu_p) \\
  &                                      +\bigg( - c\mu_pf'(\xi^*)\theta^2_2 e^{-iw^*}+ c\mu_m\mu_p \theta_2   +\frac{1}{2}g''(r^*) e^{-2iw^*}\\
   &                                              + cf'(\xi^*)g'(r^*) \theta_2e^{-2iw^*}
                                                   - c\mu_mg'(r^*) e^{-iw^*}\bigg)(\epsilon^*f'(\xi^*)e^{-i2w^*})   \bigg],\\
a_2= &\frac{\epsilon^*}{ (i2w^*+\epsilon^*\mu_m)(i2w^*+\epsilon^*\mu_p) -{\epsilon^*}^2f'(\xi^*)g'(r^*)e^{-i4w^*} }\\
&\times\bigg[  \left(  (c\mu_m^2-2c\mu_mf'(\xi^*)\theta_2e^{-iw^*} +\frac{1}{2}(f''(\xi^*)+2cf'^2(\xi^*))\theta_2^2e^{-i2w^*} \right)(\epsilon^*g'(r^*)e^{-i2w^*} ) \\
  &                                        +\bigg( - c\mu_pf'(\xi^*)\theta^2_2 e^{-iw^*}+ c\mu_m\mu_p \theta_2   +\frac{1}{2}g''(r^*) e^{-2iw^*}\\
   &                                              + cf'(\xi^*)g'(r^*) \theta_2e^{-2iw^*}
                                                   - c\mu_mg'(r^*) e^{-iw^*}\bigg)(i2w^*+\epsilon^*\mu_m) \bigg]. 
\end{aligned} 
\right.                                       
\end{align}
Equating the coefficients of the $A\bar{A}$ terms from both sides of (\ref{eqn-5-12}), we obtain,
\begin{align*}
\left\{
\begin{aligned}
\epsilon^*\mu_mb_1-\epsilon^*f'(\xi^*)b_2 =&\,\, 2\epsilon^*c\mu_m^2-2\epsilon^*c\mu_mf'(\xi^*)(\bar{\theta_2}e^{iw^*}+\theta_2e^{-iw^*})\\
&\,\,+\epsilon^*(f''(\xi^*)+2cf'^2(\xi^*))\theta_2\bar{\theta_2},\\
-\epsilon^*g'(r^*)b_1+\epsilon^*\mu_pb_2 =&\,\, -\epsilon^*c\mu_pf'(\xi^*)\theta_2\bar{\theta_2}(e^{iw^*}+e^{-iw^*})+\epsilon^*c\mu_m\mu_p(\theta_2+\bar{\theta_2})\\
                                                                 &+\epsilon^*g''(r^*)+\epsilon^*cf'(\xi^*)g'(r^*)(\theta_2+\bar{\theta_2})-\epsilon^*c\mu_mg'(r^*)(e^{iw^*}+e^{-iw^*}).
 \end{aligned}
 \right.
\end{align*}Noticing that $\theta_2=  \frac{e^{i w^*}(i w^*+\epsilon^*\mu_m)}{\epsilon^* f'(\xi^*)}$ and hence $\theta_2e^{-iw^*}+\bar{\theta}_2e^{iw^*}=\frac{2\mu_m}{f'(\xi^*)}$, we have
\begin{align*}
\left\{
\begin{aligned}
b_1=&\frac{1}{{\epsilon^*}^2f'^2(\xi^*)(\mu_m\mu_p-f'(\xi^*)g'(r^*))} \bigg[
\mu_pf''(\xi^*){w^*}^2+\mu_pf''(\xi^*){\epsilon^*}^2\mu_m^2+2f'^2(\xi^*)c\mu_p{w^*}^2\\
&-2f'^2(\xi^*)c\mu_p{w^*}^2\cos w^*-2c(f'^3(\xi^*)g'(r^*)+\mu_m\mu_pf'^2(\xi^*)){\epsilon^*}{w^*}\sin w^*\\
& +f'^3(\xi^*)g''(r^*){\epsilon^*}^2 \bigg],\\
b_2=&\frac{1}{{\epsilon^*}^2f'^2(\xi^*)(\mu_m\mu_p-f'(\xi^*)g'(r^*))} \bigg[
\mu_mf'^2(\xi^*)g''(r^*){\epsilon^*}^2+ f''(\xi^*)g'(r^*){w^*}^2\\
& +f''(\xi^*)g'(r^*)\mu_m^2{\epsilon^*}^2+2cf'^2(\xi^*)g'(r^*){w^*}^2-2c\mu_m\mu_pf'(\xi^*){w^*}^2\cos w^*\\
&-2c(\mu_mf'^2(\xi^*)g'(r^*)+\mu_m^2\mu_pf'(\xi^*)){\epsilon^*}{w^*}\sin w^* \bigg].
\end{aligned}  
\right.
\end{align*}
%\begin{align*}
%a_0&=\det\begin{bmatrix} \epsilon^*\mu_m+i2w^* & -\epsilon^*f'(0)e^{-i2w^*}\\
%                                       \epsilon^*g'(0)e^{-i2w^*} &\epsilon^*\mu_p+i2w^*
%              \end{bmatrix}\\
%       &=(\epsilon^*\mu_m+i2w^*)(\epsilon^*\mu_p+i2w^*)-(\epsilon^*)^2f'(0)g'(0)e^{-i4w^*}
%\end{align*}

%Cramer's Rule:
%\[a_1=\frac{a_{11}}{a_0}, \quad a_2=\frac{a_{22}}{a_0}\]
%\begin{align*}
%&a_{11}\\
%=&\det\begin{bmatrix}
%2c\mu_m^2-4c\mu_m\frac{f'(0)(i w^*+\epsilon^*\mu_m)}{\epsilon^*}\\
%+(f''(0)+2c(f'(0))^2)\frac{(i w^*+\epsilon\mu_m)^2}{(\epsilon^*f'(0))^2} &
%-\epsilon^* f'(0) e^{-i2w^*} \\\\
%-c\mu_p\frac{(i w^*+\epsilon^*\mu_m)^2}{\epsilon^*f'(0)}e^{i w^2}+2c\mu_m\mu_p\frac{i w^*+\epsilon^*\mu_m}{\epsilon^*f'(0)}\\
%+g''(0)e^{-2i w^*}+2cf'(0)g'(0)\frac{i w^*+\epsilon^*\mu_m}{\epsilon^*f'(0)}e^{\i w^*} &
%\epsilon^*\mu_p+i2w^* \\
%               \end{bmatrix}\\
%          =&2c\mu_m^2(\epsilon^*\mu_p+i2w^*)+4c\mu_m(\epsilon^*\mu_p+i2w^*)(i w^*+\epsilon^*\mu_m)\cdot \frac{f'(0)}{\epsilon^*}\\
%           &+(f''(0)+2c(f'(0))^2)\frac{(i w^*+\epsilon^*\mu_m)^2}{(\epsilon^*f'(0))^2}(\epsilon^*\mu_p+i2w^*)-c\mu_p(i w^*+\epsilon^*\mu_m)^2e^{-i w^*}\\
%           &+2\mu_m\mu_p(i w^*+\epsilon^*\mu_m)e^{-i w^*}+g''(0)\epsilon^*f'(0)e^{-4i w^*}+2cf'(0)g'(0)(i w^*+\epsilon^*\mu_m)e^{-3i w^*}
%\end{align*}

Substituting the expressions of $x_1$ and $x_2$ into (4.9), we obtain,
\begin{align}\label{eqn-s-514}
&D_0x_3-\epsilon^*Mx_3-\epsilon^*Nx_3(T_0-1,\,T_2)\notag\\
=&-D_2(A(T_2)\theta e^{i w^*T_0}+cc)-\epsilon^*ND_2(\theta A(T_2)e^{i w^*(T_0-1)}+cc)\notag\\
\quad\quad &+\delta (M(A(T_2)\theta e^{i w^*T_0}+cc))+N(A(T_2)\theta e^{i w^*(T_0-1)}+cc)\notag\\
  & \quad\quad+\frac{\epsilon^*}{6}
        \begin{bmatrix*}[l]
        -6c^2\mu_m^3u_1^3+18c^2\mu_m^2f'(\xi^*)u_1^2v_1(T_0-1,\,T_2)\\
        \quad\quad-(6c\mu_m f''(\xi^*)+18c^2\mu_m f'^2(\xi^*))u_1 v_1^2(T_0-1,T_2)\\
        \quad\quad+(f'''(\xi^*)+6cf''(\xi^*)f'(\xi^*)+6c^2f'^3(\xi^*))v_1^3(T_0-1,\,T_2)\\
        \hdashline
        g'''(r^*)  u_1^3(T_0-1, T_2)-6c^2\mu_m^2\mu_p  u_1^2  v_1 +6c^2\mu_m^2g'(r^*) u_1^2  u_1(T_0-1, T_2) \\
          \quad\quad  -3c\mu_mg''(r^*)u_1 u_1^2(T_0-1, T_2)\\
          \quad\quad +3cg''(r^*)f'(\xi^*)u_1^2(T_0-1, T_2)v_1(T_0-1, T_2) \\
       \quad\quad  
       -3(c\mu_p  f''(\xi^*)+2c^2\mu_pf'^2(\xi^*))  v_1 v_1^2(T_0-1, T_2) \\
        \quad\quad 
        +3(cf''(\xi^*)g'(r^*)+2c^2f'^2(\xi^*)g'(r^*))u_1(T_0-1, T_2)v_1^2(T_0-1, T_2) \\
         \quad\quad +12c^2\mu_m\mu_p f'(\xi^*)u_1 v_1 v_1(T_0-1, T_2) \\  \quad\quad-12c^2\mu_mg'(r^*)f'(\xi^*)  u_1 u_1(T_0-1, T_2)v_1(T_0-1, T_2)\\
        \end{bmatrix*}\notag\\
        &\quad\quad +\frac{\epsilon^*}{2} 
        \begin{bmatrix*}[l]
        2c\mu_m^2(u_1u_2+u_2u_1)-4c\mu_m f'(\xi^*)(u_2v_1(T_0-1, T_2)+u_1v_2(T_0-1, T_2))\\
      \quad\quad  +2(f''(\xi^*)+2cf'^2(\xi^*)))v_1(T_0-1, T_2)v_2(T_0-1)\\
            \hdashline
       -2c\mu_p f'(\xi^*)(v_1v_2(T_0-1, T_2)+v_2v_1(T_0-1, T_2))+2c\mu_m\mu_p(u_2v_1+u_1v_2)\\
      \quad\quad  +2g''(r^*)u_1(T_0-1,T_2)u_2(T_0-1, T_2)\\
      \quad\quad +2c f'(\xi^*)g'(r^*)(u_1(T_0-1, T_2)v_2(T_0-1, T_2)\\
      \quad\quad+u_2(T_0-1, T_2)v_1(T_0-1, T_2)))\\
       \quad\quad -2c\mu_m g'(r^*)(u_2u_1(T_0-1, T_2)+u_1u_2(T_0-1, T_2))
        \end{bmatrix*}.
        \end{align}
Note that $iw^*$ is an eigenvalue of the homogenous system corresponding to (\ref{eqn-s-514}) and the right hand side of the nonhomogeneous system  (\ref{eqn-s-514}) has terms with $e^{iw^*T_0}$. To remove possible secular terms  in the solutions, we  seek a particular solution of the form $x_3(T_0,\,T_2)=\phi(T_2)e^{iw^*T_0}$, where the existence of $\phi$ requires  solvability     
 of the algebraic equation
 \begin{align}\label{eqn-s-515}
 \left[iw^*I-\epsilon^*M-\epsilon^*Ne^{-iw^*}\right]\phi = \chi.
 \end{align}where $\chi$ stands for the coefficients of the $e^{iw^*T_0} $ term in (\ref{eqn-s-514}). System (\ref{eqn-s-515}) is solvable for $\phi$ if and only if $\chi$ is orthogonal to
the null space of the conjugate transpose of the coefficient matrix. That is, for every vector $d$ which satisfies
\begin{align*}
 \left[-iw^*I-\epsilon^*M^T-\epsilon^*N^Te^{iw^*}\right]d = 0,
 \end{align*}    we require
  \begin{align}\label{eqn-s-516}
\bar{d}^T\chi=0,
 \end{align}which is an ordinary differential equation of $A$ and is the normal form. We can obtain $d$ as following with the condition $\bar{d}^T\theta=1$ imposed for uniqueness:
 \begin{align*}
 d=\frac{1}{-2iw^*+\epsilon^*(\mu_m+\mu_p)} 
 \begin{bmatrix}
 {-iw^*+\epsilon^*\mu_p}\\
 \epsilon^*e^{iw^*}f'(\xi^*)
 \end{bmatrix}.
 \end{align*}
 
 Bringing the expressions of $x_1$ and $x_2$ into system (\ref{eqn-s-514}), we have
the following expression for $\chi:$
\begin{align}\label{eqn-s-517}
\chi=
&- \theta A'-\epsilon^*e^{-iw^*}N\theta A'+\delta M\theta A+\delta N\theta Ae^{-iw^*}
\\
 &+\frac{\epsilon^*}{2}A^2\bar{A}
\begin{bmatrix*}[l]
&-6c^2\mu_m^3+6c^2\mu_m^2f'(\xi^*)(2\theta_2e^{-iw^*}+ \bar{\theta}_2e^{iw^*})-2c\mu_m(f''(\xi^*)+3cf'^2(\xi^*))\\
& \cdot\theta_2(\theta_2e^{-2iw^*}+ 2\bar{\theta}_2)+(f'''(\xi^*)+6cf''(\xi^*)f'(\xi^*)+6c^2f'^3(\xi^*))\theta_2^2\bar{\theta}_2e^{-iw^*}\\
\hdashline
& g'''(r^*)e^{-iw^*}-2c^2\mu_m^2\mu_p(\bar{\theta}_2+2\theta_2)+2c^2\mu_m^2g'(r^*)(2e^{-iw^*}+e^{iw^*})\\
&-c\mu_mg''(r^*)(2+e^{-2iw^*})+cg''(r^*)f'(\xi^*)(2\theta_2e^{-iw^*}+ \bar{\theta}_2e^{-iw^*})\\
&-c\mu_p(f''(\xi^*)+2cf'(\xi^*))\theta_2^2\bar{\theta}_2(2+e^{-2iw^*})+cg'(r^*)(f''(\xi^*)+2cf'^2(\xi^*))\\
&\cdot(2\bar{\theta}_2+\theta_2)\theta_2e^{-iw^*}+4c^2\mu_m\mu_pf'(\xi^*)\theta_2(\bar{\theta}_2e^{iw^*}+ \bar{\theta}_2e^{-iw^*}+\theta_2e^{-iw^*})\\
&-4c^2\mu_mf'(\xi^*)g'(r^*) (\bar{\theta}_2 + {\theta}_2+  \theta_2e^{-2iw})
\end{bmatrix*}\\
&+\epsilon^*A^2\bar{A}
\begin{bmatrix*}[l]
& 2c\mu_m^2 (a_1+b_1)-2c\mu_mf'(\xi^*)(a_1\bar{\theta}_2e^{iw}+b_1{\theta}_2e^{-iw}+b_2+a_2e^{-2iw^*}) \\
&+(f''(\xi^*)+2cf'^2(\xi^*))(a_2\bar{\theta}_2+b_2\theta_2)e^{-iw^*}\\
\hdashline
& -c\mu_pf'(\xi^*)(b_2\theta_2+a_2\bar{\theta}_2e^{-2iw^*}+b_2\theta_2e^{-iw^*}+a_2\bar{\theta}_2e^{iw^*})\\
&+c\mu_m\mu_p(b_1\theta_2+a_1\bar{\theta}_2+a_2+b_2)+g''(r^*)e^{-iw^*}(a_1+b_1)\\
&+ cf'(\xi^*)g'(r^*)(b_2e^{-iw^*}+a_2 e^{-iw^*}+b_1\theta_2e^{-iw^*}+a_1\bar{\theta}_2 e^{-iw^*})\\
&- c\mu_mg'(r^*)(b_1e^{-iw^*}+a_1 e^{iw^*}+b_1+a_1 e^{-2iw^*})
\end{bmatrix*}.
\end{align}
By (\ref{eqn-s-516}) and  (\ref{eqn-s-517}) and noticing that   $\bar{d}^T\theta=1$ and
\[
iw^*\theta-\epsilon^*M\theta-\epsilon^*N\theta e^{-iw^*}=0,
\]
we have the following normal form:
\begin{align}
& A'(1+\epsilon^*e^{-iw^*}\bar{d}^TN\theta)\notag
\\
 &=   \frac{iw^*}{\epsilon^*}\delta A\notag\\
&+\frac{\epsilon^*}{2}A^2\bar{A}\bar{d}^T\begin{bmatrix*}[l]
&-6c^2\mu_m^3+6c^2\mu_m^2f'(\xi^*)(2\theta_2e^{-iw^*}+ \bar{\theta}_2e^{iw^*})-2c\mu_m(f''(\xi^*)+3cf'^2(\xi^*))\\
& \cdot\theta_2(\theta_2e^{-2iw^*}+ 2\bar{\theta}_2)+(f'''(\xi^*)+6cf''(\xi^*)f'(\xi^*)+6c^2f'^3(\xi^*))\theta_2^2\bar{\theta}_2e^{-iw^*}\\
\hdashline
& g'''(r^*)e^{-iw^*}-2c^2\mu_m^2\mu_p(\bar{\theta}_2+2\theta_2)+2c^2\mu_m^2g'(r^*)(2e^{-iw^*}+e^{iw^*})\\
&-c\mu_mg''(r^*)(2+e^{-2iw^*})+cg''(r^*)f'(\xi^*)(2\theta_2e^{-iw^*}+ \bar{\theta}_2e^{-iw^*})\\
&-c\mu_p(f''(\xi^*)+2cf'(\xi^*))\theta_2^2\bar{\theta}_2(2+e^{-2iw^*})+cg'(r^*)(f''(\xi^*)+2cf'^2(\xi^*))\\
&\cdot(2\bar{\theta}_2+\theta_2)\theta_2e^{-iw^*}+4c^2\mu_m\mu_pf'(\xi^*)\theta_2(\bar{\theta}_2e^{iw^*}+ \bar{\theta}_2e^{-iw^*}+\theta_2e^{-iw^*})\\
&-4c^2\mu_mf'(\xi^*)g'(r^*) (\bar{\theta}_2 + {\theta}_2+  \theta_2e^{-2iw})
\end{bmatrix*}\notag\\
&+\epsilon^*A^2\bar{A}\bar{d}^T
\begin{bmatrix*}[l]
& 2c\mu_m^2 (a_1+b_1)-2c\mu_mf'(\xi^*)(a_1\bar{\theta}_2e^{iw}+b_1{\theta}_2e^{-iw}+b_2+a_2e^{-2iw^*}) \\
&+(f''(\xi^*)+2cf'^2(\xi^*))(a_2\bar{\theta}_2+b_2\theta_2)e^{-iw^*}\\
\hdashline
& -c\mu_pf'(\xi^*)(b_2\theta_2+a_2\bar{\theta}_2e^{-2iw^*}+b_2\theta_2e^{-iw^*}+a_2\bar{\theta}_2e^{iw^*})\\
&+c\mu_m\mu_p(b_1\theta_2+a_1\bar{\theta}_2+a_2+b_2)+g''(r^*)e^{-iw^*}(a_1+b_1)\\
&+ cf'(\xi^*)g'(r^*)(b_2e^{-iw^*}+a_2 e^{-iw^*}+b_1\theta_2e^{-iw^*}+a_1\bar{\theta}_2 e^{-iw^*})\\
&- c\mu_mg'(r^*)(b_1e^{-iw^*}+a_1 e^{iw^*}+b_1+a_1 e^{-2iw^*}) \label{eqn-s-518}
\end{bmatrix*}.
\end{align}We arrive at,
\begin{theorem}\label{Th34}Assume that (B1), (B2) and (B3) hold. The normal form of system~(\ref{SDDE-general-1}) near its equilibrium  is   (\ref{eqn-s-518}).
\end{theorem}
We remark that if the state-dependent translation time $T$ is not included in the diffusion time $\tau$, namely, $c=0$, the normal form at (\ref{eqn-s-518}) becomes the normal form for system (\ref{SDDE-general-1}) with constant delay $\epsilon$:
\begin{align}\label{eqn-s-519}
 A'
= &\frac{iw^* e^{iw^*}\delta}{\epsilon^*(e^{iw^*}+\epsilon^*
\bar{d}^TN\theta)}A\notag\\
& +\frac{{\epsilon^*}}{( {e^{iw^*}}+{\epsilon^*}\bar{d}^TN\theta)}\bar{d}^T
\begin{bmatrix*}[l]
&f''(\xi^*) (a_2\bar{\theta}_2+b_2\theta_2)+ \frac{1}{2}f'''(\xi^*)\theta_2^2\bar{\theta}_2 \\ 
& g''(r^*)(a_1+b_1)+ \frac{1}{2}g'''(r^*)
\end{bmatrix*}A^2\bar{A}.
\end{align} 

%\begin{theorem} Consider system~(\ref{SDDE-general-1}) where $f,\,g:\mathbb{R}\rightarrow\mathbb{R}$ are three times continuously differentiable; $\mu_m,\,\mu_p$ and $c$ are positive constants. Let $(r^*,\,\xi^*)$ be an equilibrium. 
%\begin{enumerate}
%\item[i)]If $
% \mu_m\mu_p\geq -f'(\xi^*)g'(r^*)\geq 0,$  then the equilibrium of sytem~(\ref{SDDE-general-1}) is asymptotically stable.
%\item[ii)] If  
%$
% \mu_m\mu_p< -f'(\xi^*)g'(r^*),
% $ then there exists a unique $\epsilon_0\in (0,\,+\infty)$ such that 
% \begin{enumerate}
% \item[a$)$] for every $\epsilon\in (0,\,\epsilon_0)$,  the equilibrium $(r^*,\,\xi^*)$  is asymptotically stable;
% \item[b$)$]  for every $\epsilon\in (\epsilon_0,\,+\infty)$,  the equilibrium $(r^*,\,\xi^*)$  is unstable, and there exists a sequence $\{\epsilon_k\}_{k=0}^\infty$ of critical values of $\epsilon$ for which  (\ref{charac-1}) has a corresponding sequence of purely imaginary eigenvalues  $\{i(\beta(\epsilon_0)+k\pi)\}_{k=1}^\infty$ where
% \[
% \epsilon_k=\frac{\epsilon_0(\beta(\epsilon_0)+k\pi)}{\beta(\epsilon_0)}\in S_k,
% \] and 
% \begin{align*}S_0 &=\left\{\lambda\in\mathbb{C}: \mathrm{Im}{\lambda}\in (0,\,\frac{\pi}{2})\right\},\\   S_k & =\left\{\lambda\in\mathbb{C}: \mathrm{Im}{\lambda}\in ((k-\frac{1}{2})\pi,\,(k+\frac{1}{2})\pi)\right\}, \,k=1,\,2\cdots.
% \end{align*}
% \end{enumerate}
%\end{enumerate}
%\end{theorem}
%
%
 
\section { Hes1 system~(\ref{SDDE-hes1})  and numerical simulation}\label{Sec6}
In this section, we consider the prototype system~(\ref{SDDE-hes1}) and illustrate the results of section 3 by numerical simulation. In the following,  a vector is said to be positive if all coordinates are positive. To have assumptions (B1) and (B2) hold for system~(\ref{SDDE-hes1}), we show that,
\begin{theorem} Consider system~(\ref{SDDE-hes1}) with positive parameters $\mu_m$, $\mu_p$, $\alpha_m$, $\alpha_p$, $h$, $\bar{y}$, $\epsilon$ and $c$.  Suppose that $\alpha_m\leq\frac{1}{c}$. Then system~(\ref{SDDE-hes1}) has  a positive equilibrium $(x^*,\,y^*)$ and there exists a neighborhood $U$ of $(x^*,\,y^*)\subset C^1([-\alpha_0,\,0];\mathbb{R}^2)$ with $\alpha_0>0$  such that, for every  initial data for $(x,\,y)\in U$,  and initial data for $\tau$ in $(0,\,\alpha_0)$ satisfying the compatibility conditions at (\ref{compatibility}), there exists a unique positive solution $(x,\,y,\,\tau)$ on $[0,\,+\infty)$ with $\dot{x}(t)<\alpha_m$ for every $t\geq 0$.
\end{theorem}
\begin{proof} Since all parameters of  system~(\ref{SDDE-hes1}) are positive, direct computation shows that there exists a positive equilibrium. Then there exists    a neighborhood $U$ of $(x^*,\,y^*)\subset C^1([-\alpha_0,\,0];\mathbb{R}^2)$ such that every $(x,\,y)\in U$ is positive and $\dot{x}(s)<\frac{1}{c}$ for all $s\in [-\alpha_0,\,0]$.

Consider the initial value problem associated with  system~(\ref{SDDE-hes1}).  We obtain by the work of \cite{Walther2004} that the initial value problem has a unique  solution $(x,\,y,\,\tau)$ on $[0,\,t_e)$ where  $[0,\,t_e)$ is the maximal existence interval with $t_e>0$ or $t_e=\infty$ (c.f.,\cite{Driver1984}). 
% Moreover, if $t_e<\infty$, we have $\lim_{t\rightarrow t_e}|(x,\,y,\,z)(t)|=\infty$. 
%Regarding  system~(\ref{SDDE-hes1})  as a system of nonautonomous ordinary differential equations with   $\tau$ satisfying
%\[
%\dot{\tau}(t)=\frac{\dot{x}(t)-\dot{x}(t-\tau(t))}{\frac{1}{c}-\dot{x}(t-\tau(t))},
%\] where   $\dot x(-\tau_0)<\frac{1}{c}$ with $\tau_0$ the initial value of $\tau$. 

Notice that if $\tau(s)=0$ for some $s\in (0,\,t_e)$ then from the third equation of system~(\ref{SDDE-hes1}) we have $\tau(s)=\epsilon>0$ which  contradicts $\tau(s)=0$.
Therefore, we have $\tau(s)>0$ for every $s\in [0,\,t_e)$.

Next we show that  $\dot{x}(t)<\frac{1}{c}$ for every $t\in (0,\,t_e)$. Assume that there exists a minimal $s_0\in (0,\,t_e)$ such that $\dot{x}(s_0)=\frac{1}{c}$. Then by the third equation of system~(\ref{SDDE-hes1})  we have for every $t\in [0,\,s_0]$,  
\[
\dot{\tau}(t)=\frac{\dot{x}(t)-\dot{x}(t-\tau(t))}{\frac{1}{c}-\dot{x}(t-\tau(t))}<1.
\] It follows that for every $t\in [0,\,s_0]$ we have $-\tau_0<t-\tau(t)<t$. Then we have
\[
\dot{x}(s)=-\mu_m x(s)+\frac{\alpha_m}{1+\left(\frac{y(s-\tau(s))}{\bar{y}}\right)^h}<\frac{\alpha_m}{1+\left(\frac{y(s-\tau(s))}{\bar{y}}\right)^h}<\alpha_m,
\]which leads to $\frac{1}{c}<\alpha_m$. This is a contradiction. Therefore, $\dot{x}(t)<\frac{1}{c}$ for every $t\in (0,\,t_e)$.

Next we show that $x(t)$ and $y(t)$  are positive on $[0,\,t_e)$. Suppose, for contradiction, that there exists $t_0\in (0,\,t_e)$ which is the minimal time such that $x(t)$ or $y(t)$   vanishes. If $x(t_0)=0$, then we have
\[
\dot{x}(t_0)=-\mu_m x(t_0)+\frac{\alpha_m}{1+\left(\frac{y(t_0-\tau(t_0))}{\bar{y}}\right)^h}=\frac{\alpha_m}{1+\left(\frac{y(t_0-\tau(t_0))}{\bar{y}}\right)^h}>0,
\]which means that there exists $t_0^*\in (0,\,t_0)$ such that $x(t_0^*)<0$. This contradicts the minimality of $t_0$.

 If $y(t_0)=0$, then we have
\[
\dot{y}(t_0)=-\mu_p y(t_0)+\alpha_px(t_0-\tau(t_0))=\alpha_px(t_0-\tau(t_0))>0,
\]which means that there exists $t_0^{**}\in (0,\,t_0)$ such that $y(t_0^{**})<0$. This contradicts the minimality of $t_0$. That is, $x(t)$, $y(t)$ and $\tau(t)$ are positive on $[0,\,t_e)$ and for every $t\in [0,\,t_e)$ we have
\begin{align}\label{est-x}
\dot{x}(t)=  -\mu_m x(t)+\frac{\alpha_m}{1+\left(\frac{y(t-\tau(t))}{\bar{y}}\right)^h}<\frac{\alpha_m}{1+\left(\frac{y(t-\tau(t))}{\bar{y}}\right)^h}<\alpha_m.
\end{align}It follows that 
\begin{align}\label{est-y}
\dot{y}(t)=  -\mu_p y(t)+\alpha_px(t-\tau(t))<\alpha_px(t-\tau(t))<\alpha_p(\alpha_m t+x(0)).
\end{align} By (\ref{est-x}) and (\ref{est-y}), and by  the third equation of  system~(\ref{SDDE-hes1}), $(x,\,y,\,\tau)$ is uniformly bounded on $[-\alpha_0,\,t_e)$ and hence extendable beyond $t_e$, which contradicts the maximality of $t_e$ and we have $t_e=\infty$.
\qed
\end{proof}
Now we turn to assumption (B3). Let $\mu_m=0.03$, $\mu_p=0.04$, $\alpha_m=35$, $\alpha_p=10$, $\bar{y}=1200$, $h=5$. Then we have the equilibria $(r^*,\,\xi^*)=(11.97050076, \,2992.625189)$, $\omega^*=0.47038322$. We have $f'(\xi^*)=-0.00059384374$, $g'(r^*)=10$. Then we have
\[
\mu_m\mu_p=1.2\times 10^{-3}  <5.9384374\times 10^{-3}=-f'(\xi^*)g'(r^*).
\]
Then by Lemma~\ref{lemma-1} and by Theorem~\ref{thm3-3}, there exists a unique $\epsilon_0=6.86216245$, such that for every $\epsilon\in (0,\,\epsilon_0)$, the equilibria $(r^*,\,\xi^*)$ is stable; for every $\epsilon\in [\epsilon_0,\,\infty)$ the equilibria $(r^*,\,\xi^*)$ is unstable and   $\epsilon_0$ is a critical value of $\epsilon$ for which system~(\ref{SDDE-hes1}) undergoes Hopf bifurcation. For $\epsilon=\epsilon_0$ and $\lambda=2iw^*$, we have 
\begin{align*}
& (\lambda+\epsilon\mu_m)(\lambda+\epsilon\mu_p)-\epsilon^2f'(\xi^*)g'(r^*)e^{-2\lambda}\big|_{\,\lambda=2iw^*,\,\epsilon=\epsilon_0}\\
= & -0.9140361052+0.1856539388\,i\neq 0.
\end{align*}That is, (B3) holds. By Theorem~\ref{Th34}, the normal form of system~(\ref{SDDE-hes1}) near  $(r^*,\,\xi^*)$ is,
\begin{align*}
A'= &(0.01841158248+0.04829902976\,i)\delta A \\
& +
 \left((2.114544332\,c^2+0.0008578251748\,c-0.001233336633)\right.\\
& \left.-i(1.469928514\, c^2+0.002534237744\,c-0.003599996653)\right)A^2\bar{A},
\end{align*}
where the real part of the coefficient of the cubic term $ A^2\bar{A}$ is 
\[
2.114544332\,c^2+0.0008578251748\,c-0.001233336633,
\]which is negative for $0<c<c_0:=0.02394886242,$ and is positive for $c>c_0.$

If $0<c<c_0$, system~(\ref{SDDE-hes1}) undergoes supercritical Hopf bifurcation near  $(r^*,\,\xi^*)$ as $\epsilon$ crosses the critical value $\epsilon_0$. Namely, if $0<\epsilon<\epsilon_0$, the equilibrium is stable (see Figure~\ref{sup-fig-periodic}(a));  if $\epsilon>\epsilon_0$, there is a stable periodic solution near  $(r^*,\,\xi^*)$ (see Figure~\ref{sup-fig-periodic}(b)).

If $c>c_0$, system~(\ref{SDDE-hes1}) undergoes subcritical Hopf bifurcation near  $(r^*,\,\xi^*)$ as $\epsilon$ crosses the critical value $\epsilon_0$. Namely, if $0<\epsilon<\epsilon_0$, the equilibrium is stable  (see Figure~\ref{sub-fig-periodic}(a))  and  there is a unstable periodic solution near  $(r^*,\,\xi^*)$. See Figure~\ref{sub-fig-periodic}(a), where the equilibrium  $(r^*,\,\xi^*)=(11.97050076, \,2992.625189)$ is asymptotically stable with $\epsilon=\epsilon_0-\delta$, $c=c_0+0.001$ (the solid curve); when initial value is far enough from the equilibrium   $(r^*,\,\xi^*)=(11.97050076, \,2992.625189)$, solution with nonpositive initial values  (the dashed curve) may converge to a negative equilibrium.
  If $\epsilon>\epsilon_0$, the equilibrium  $(r^*,\,\xi^*)$ is unstable  (see Figure~\ref{sub-fig-periodic}(b)). 
\begin{figure}[H]
%\begin{center}
\scalebox{0.54}{
\begin{subfigure}[b]{0.45\textwidth}
 \includegraphics[angle=0]{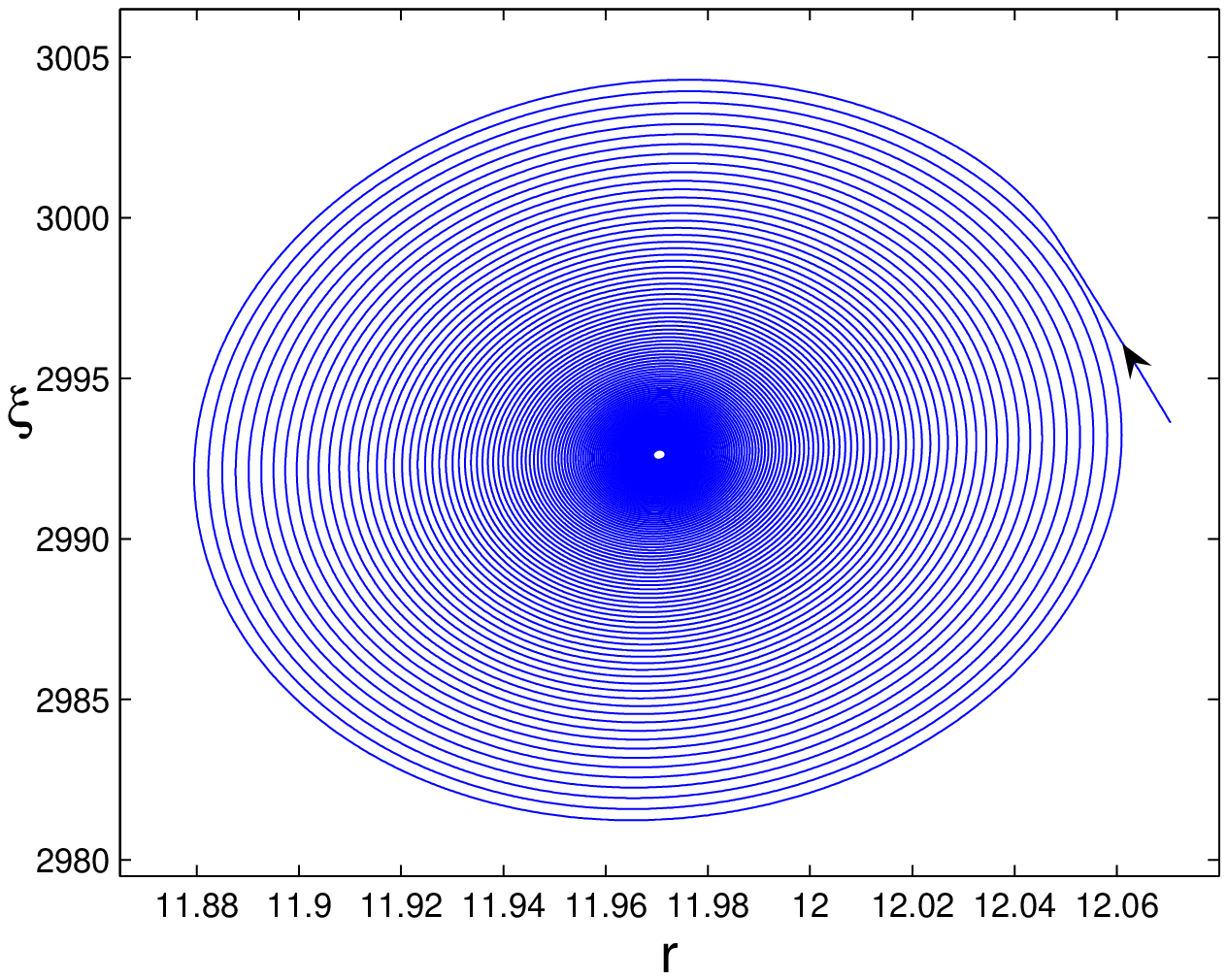}
 \subcaption{}
 \end{subfigure}
 }\hspace*{8em}
 \scalebox{0.54}{
 \begin{subfigure}[b]{0.45\textwidth}
 \includegraphics[angle=0]{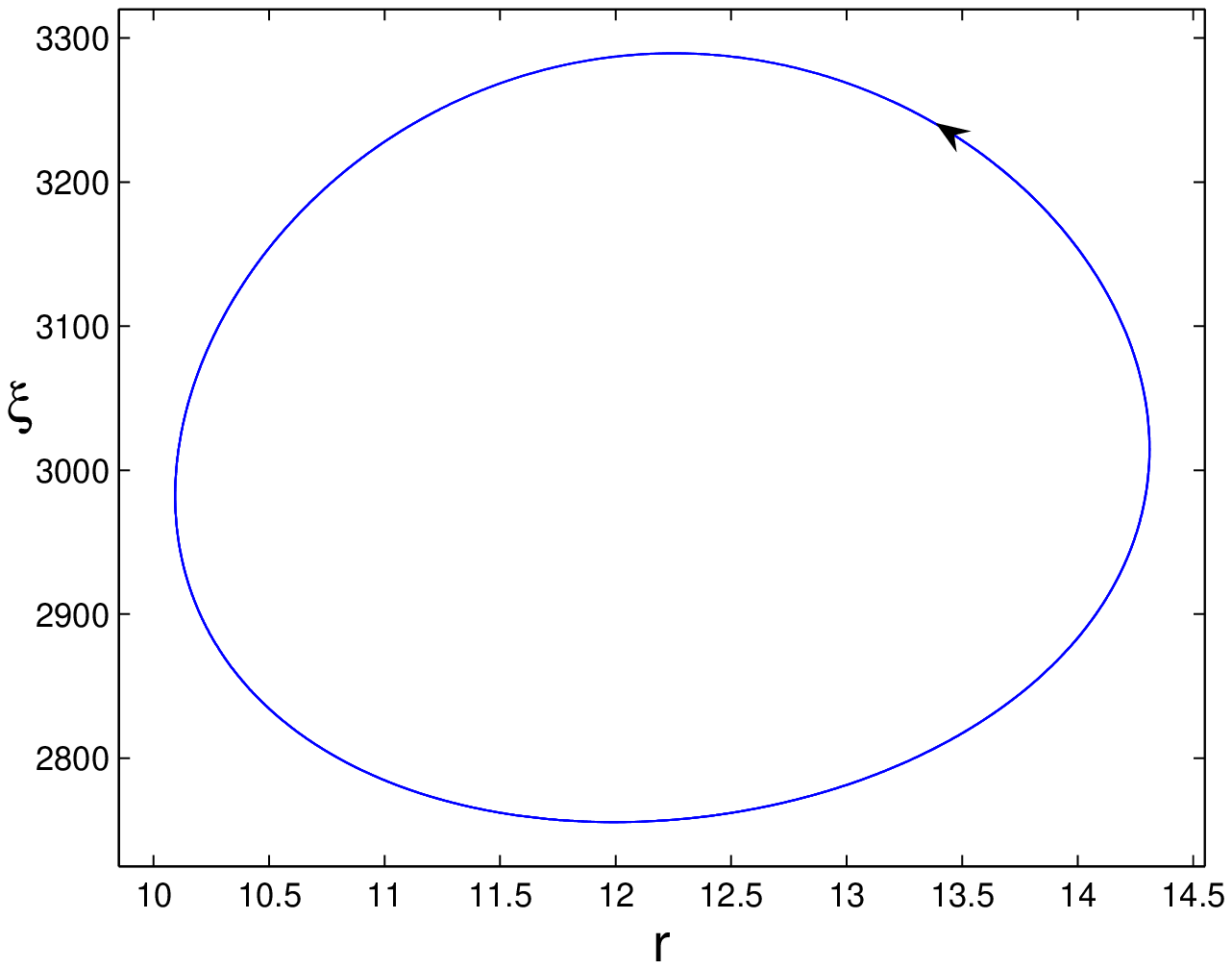}
 \subcaption{}
 \end{subfigure}
 }
 %\end{center}
 \caption{(a) Equilibrium  $(r^*,\,\xi^*)=(11.97050076, \,2992.625189)$ is stable with $\epsilon=\epsilon_0-\delta$, $c=0.01<c_0$ with $\delta=0.1$; (b) periodic solution appears at $\epsilon=\epsilon_0+\delta$, $c=0.01<c_0$.}
\label{sup-fig-periodic}
 \end{figure}
 \begin{figure}[H]
%\begin{center}
\scalebox{0.54}{
\begin{subfigure}[b]{0.45\textwidth}
 \includegraphics[angle=0]{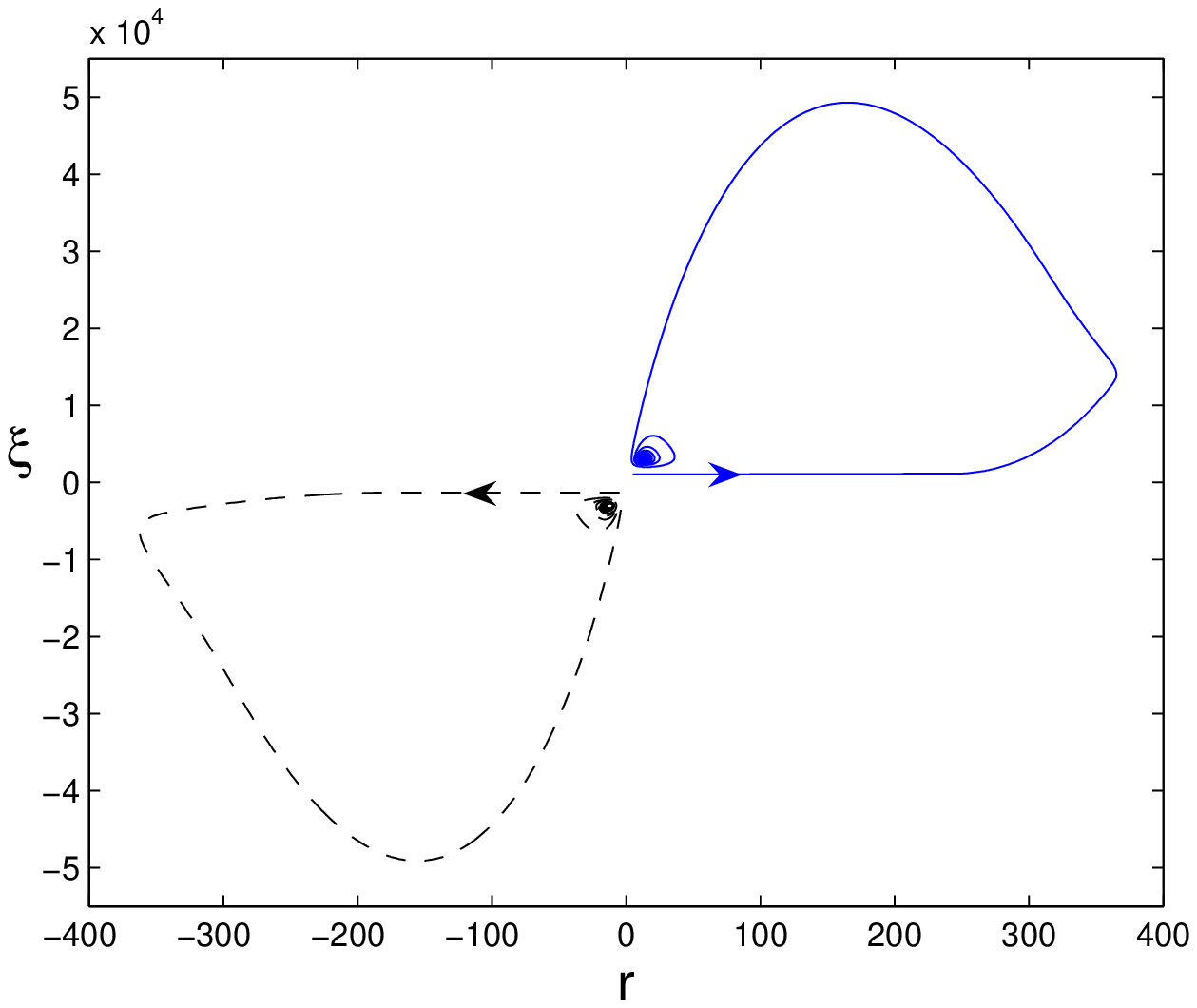}
 \subcaption{}
 \end{subfigure}
 }\hspace*{8.25em}
 \scalebox{0.54}{
 \begin{subfigure}[b]{0.45\textwidth}
 \includegraphics[angle=0]{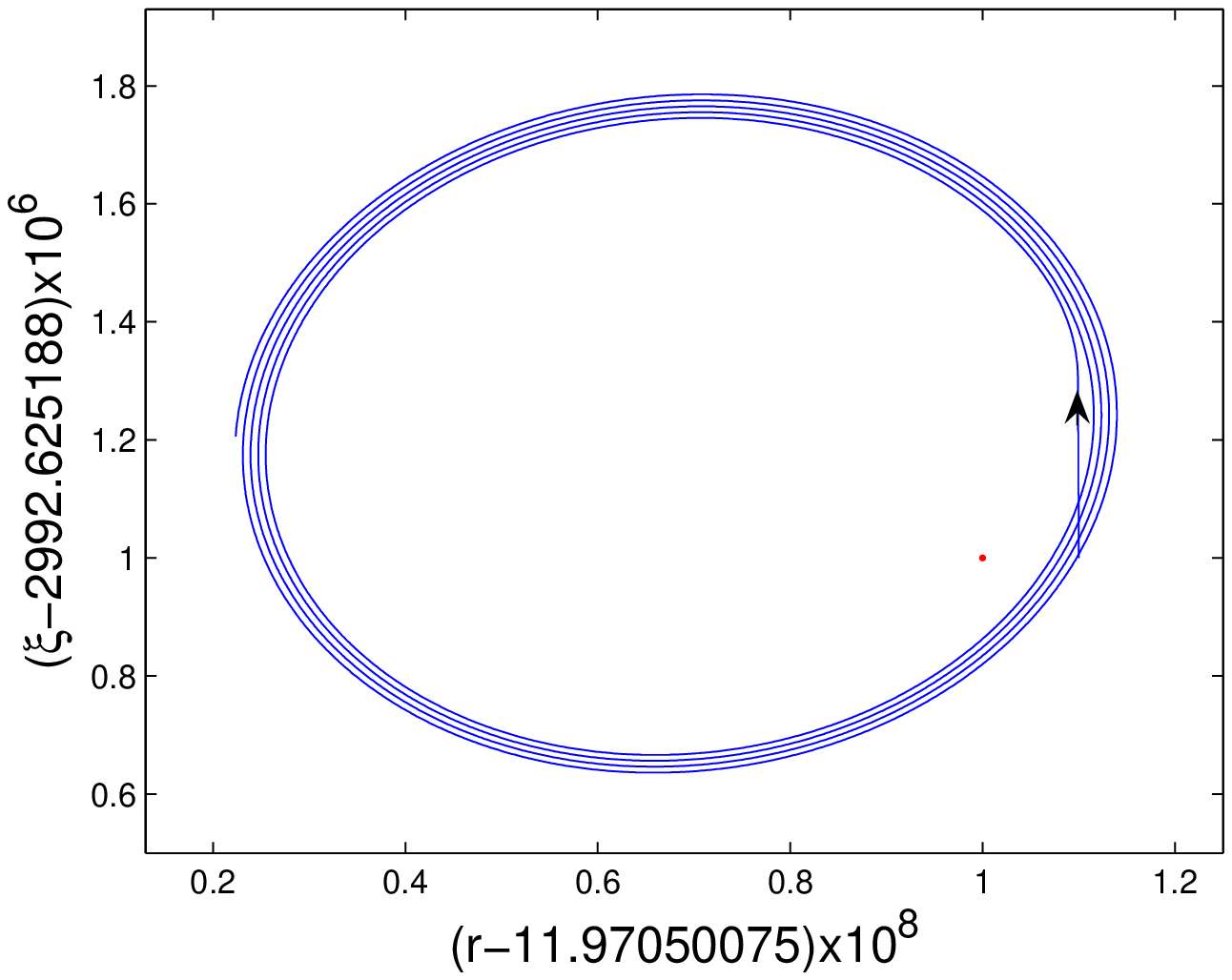}
 \subcaption{}
 \end{subfigure}
 }
 %\end{center}
 \caption{(a) Equilibrium  $(r^*,\,\xi^*)=(11.97050076, \,2992.625189)$ is asymptotically stable with $\epsilon=\epsilon_0-\delta<\epsilon_0$, $c=c_0+0.001$ (see the solid curve); when initial value is far enough from the equilibrium   $(r^*,\,\xi^*)=(11.97050076, \,2992.625189)$, solution may converge to another equilibrium  (see the dashed curve). Subcritical bifurcation occurs at $\epsilon_0$ with $0<c<c_0$. (b)  If $\epsilon>\epsilon_0$ and $c=c_0+0.001$, the equilibrium  $(r^*,\,\xi^*)$ is unstable. }
\label{sub-fig-periodic}
 \end{figure}

\section{Concluding remarks}\label{Sect5}
Starting from  model (\ref{SDDE-general-1}) of differential equations with threshold type state-dependent delay,
we obtained  model (\ref{tranformed-general}) with constant delay equivalent to (\ref{SDDE-general-1})  under assumptions (B1), (B2) and (B3), by using a time domain transformation.  System~ (\ref{tranformed-general})  provides a possibility to investigate the dynamics of the original system near the equilibrium points. We remark that   the normal form  (\ref{eqn-s-518}) of the general   type of system~(\ref{tranformed-general})  obtained by using the method of multiple time scales is also applicable to the following type of system:
   \begin{align}\label{tranformed-general-c-1}
\left\{
\begin{aligned}
\frac{1}{ \epsilon} \frac{\mathrm{d}r}{\mathrm{d} \eta} & ={-\mu_m r(\eta)+f(\xi(\eta-1))}\\
 \frac{1}{ \epsilon}  \frac{\mathrm{d}\xi}{\mathrm{d} \eta} & = {-\mu_p \xi(\eta)+g(r(\eta-1))},
\end{aligned}
\right.
\end{align}
  which is equivalent to system~(\ref{SDDE-general-c-0}). Even though many specific types of models of genetic regulatory dynamics have been developed in recent years (see  \cite{WAN2009464}, among many others), the general normal form we developed applies to a broad class of models.  
  
 Modeling the state-dependent delay of the genetic regulatory dynamics gives rise to the parameters $c>0$ and $\epsilon>0$ in  model (\ref{tranformed-general}) with constant delay, where $c$ takes account of the state-dependent fluctuation of the diffusion time for  homogenization of the substances produced in the regulatory network and $\epsilon$ measures the basal diffusion time. Since $c$ is the coefficient of the concentration disparity $x(t)-x(t-\tau)$, it is an analogue of the diffusion coefficients in \cite{Mahaffy-Busenberg}. The basal diffusion time $\epsilon$ is an analogue of the active macromolecular transport modeled as a time delay in  \cite{Mahaffy-Busenberg}. 
  
  From the analysis of  model (\ref{tranformed-general}), we find that the characteristic equation is independent of   $c$ while  the basal diffusion time $\epsilon$ determines (with other parameters fixed) the stability of the equilibrium state and the approximate period of the bifurcating periodic solutions. However, the state-dependent effect represented by $c$  may give rise to both supercritical and subcritical Hopf bifurcations, while the former has been observed in \cite{WAN2009464}, the latter is rarely seen in the literature. Note that  the work of \cite{WAN2009464} is a theoretic analysis with normal form computation for the model with constant delay considered in \cite{Smolen1999a}, where the nonlinear feedbacks are Holling type III functions and  only numerical analysis was conducted.
  
  A subcritical Hopf bifurcation usually means a sudden change of the dynamics when the system state is far enough from the stable equilibrium state. In terms of the model for the state-dependent delay
 $
 \tau(t)=\epsilon+c(x(t)-x(t-\tau(t))), 
 $ the subcritical Hopf bifurcation occurs when the contribution of homogenization to the diffusion time is strong enough and the basal diffusion time $\epsilon$ is less than its critical value. This observation  may shed lights on the mutation phenomenon in certain genetic regulatory dynamics. 
 
 We remark that the numerical results in Section~\ref{Sec6} is consistent with the observations from  \cite{Mahaffy-Busenberg}   described in the Section~\ref{Sec1}: when the analogous diffusion coefficient $c$ is large enough, the Hes1 system undergoes subcritical Hopf bifurcation with respect to $\epsilon$ and cannot sustain stable periodic oscillations; when the analogous diffusion coefficient $c$ is small, the Hes1 system undergoes supcritical Hopf bifurcation with respect to $\epsilon$ and can sustain stable periodic oscillations. The  state-dependent delay provides an approach to take into account of both the diffusive and active macromolecular transports without the complication of involving spatial information which typically will lead to reaction-diffusion equations.

 We also remark that the attempt to model the state-dependent delay with an inhomogeneous linear function of the state variables   reveals that the  concentration gradients in the genetic regulatory dynamics   may change the Hopf bifurcation direction modeled with constant delays.  Since diffusion of  substances in live cell may be confined or otherwise free and is a complex process \cite{Wu1430}, improved models for the state-dependent delay in the diffusion process may provide better means for understanding the underlying mechanisms of genetic regulatory dynamics.  For instance, what if the stationary state of $\tau$ depends on $c$ so that   $c$ also changes the stability of the equilibia?  It also remains to be investigated that to which extent the model of regulatory dynamics with linear type of state-dependent delay reflects that with a nonlinear one. 

 \section*{Acknowledgement}The author would like to thank an anonymous referee for the detailed and constructive comments which greatly improved the paper.
\bibliographystyle{siam}

\end{document}